\documentclass[11pt]{amsart}
\usepackage{amsmath,amsthm,hyperref,enumitem,mathtools,esint}
\usepackage[alphabetic,initials,nobysame]{amsrefs}
\usepackage{color,todonotes}

\title[Mappings of finite distortion]{Mappings of finite distortion on metric surfaces}

\author{Damaris Meier}
\address[Damaris Meier]{Department of Mathematics\\ University of Fribourg\\ Chemin du Mus\'ee 23\\ 1700 Fribourg, Switzerland.}
\email{damaris.meier@unifr.ch}

\author{Kai Rajala}
\address[Kai Rajala]{Department of Mathematics and Statistics, University of Jyväskylä, P.O. Box 35 (MaD), 40014 University of Jyväskylä, Finland}
\email{kai.i.rajala@jyu.fi}

\thanks{The first-named author is partially supported by UniFr Doc.Mobility Grant DM-22-10.}
\date{\today}

\numberwithin{equation}{section}

\newtheorem{thm}{Theorem}[section]
\newtheorem{prop}[thm]{Proposition}

\newtheorem{lemma}[thm]{Lemma}
\newtheorem{corollary}[thm]{Corollary}

\theoremstyle{remark}
\newtheorem{rmk}[thm]{Remark}

\theoremstyle{definition}

\newcommand{\R}{\mathbb{R}}
\newcommand{\N}{\mathbb{N}}
\newcommand{\D}{\mathbb{D}}
\newcommand{\Nloc}{N^{1,2}_{\text{loc}}}
\newcommand{\Lloc}{L_{\text{loc}}}
\newcommand{\hm}{{\mathcal H}}
\newcommand{\lip}{\mathrm{Lip}}
\newcommand{\dist}{\operatorname{dist}}

\renewcommand{\mod}{\operatorname{Mod}}

\DeclareMathOperator{\id}{id}
\DeclareMathOperator{\md}{md}
\DeclareMathOperator{\ap}{ap}
\DeclareMathOperator{\apmd}{ap\,md}

\begin{document}

\begin{abstract}
We investigate basic properties of \emph{mappings of finite distortion} $f:X \to \R^2$, where $X$ is any \emph{metric surface}, i.e., metric space homeomorphic to a planar domain with locally finite $2$-dimensional Hausdorff measure. We introduce \emph{lower gradients}, which complement the upper gradients of Heinonen and Koskela, to study the distortion of non-homeomorphic maps on metric spaces. 

We extend the Iwaniec-\v{S}ver\'ak theorem to metric surfaces: a non-constant $f:X \to \R^2$ with locally square integrable upper gradient and locally integrable distortion is continuous, open and discrete. We also extend the Hencl-Koskela theorem by showing that if $f$ is moreover injective then $f^{-1}$ is a Sobolev map. 
\end{abstract}

\maketitle

\section{Introduction}
\subsection{Background}
Let $\Omega \subset \R^2$ be a domain. We say that map $f:\Omega \to \R^2$ in the Sobolev space $W^{1,2}_\text{loc}(\Omega,\R^2)$ has \emph{finite distortion} if there is a measurable function $K:\Omega \to [1,\infty)$ so that  
\begin{equation} \label{eq:distortion_ineq}
||Df(x)||^2 \leq K(x) J_f(x) \quad \text{for a.e. } x \in \Omega. 
\end{equation} 
Here $||Df(x)||$ and $J_f(x)$ are the operator norm and determinant of $Df(x)$, respectively. 

If $K(x)=1$ for almost every $x \in \Omega$, then \eqref{eq:distortion_ineq} is valid if and only if $f$ is complex analytic. The basic topological properties of non-constant analytic functions are \emph{continuity}, \emph{openness} and \emph{discreteness} (the preimage of every point is discrete in $\Omega$). 

By \emph{Sto\"ilow factorization} (see \cite{AIM09}*{Chapter 5.5}, \cite{LP20}) non-constant \emph{quasiregular maps}, i.e., maps $f$ satisfying \eqref{eq:distortion_ineq} with constant function $K(x)=K \geq 1$, admit a factorization $f=g\circ h$, where $h$ is a quasiconformal homeomorphism and $g$ is analytic. In particular, every such $f$ is also continuous, open and discrete. 

In \cite{IwaSve93} Iwaniec and \v{S}ver\'ak showed that boundedness of $K(x)$ may be replaced with local integrability. 

\begin{thm}[Iwaniec-\v{S}ver\'ak theorem] \label{thm:IS}
Suppose $f \in W^{1,2}_\text{loc}(\Omega,\R^2)$ is non-constant and satisfies \eqref{eq:distortion_ineq} for some locally integrable $K(x)$. Then $f$ is continuous, open and discrete. 
\end{thm}
The assumption on $K(x)$ is essentially best possible (see \cite{Bal81} and \cite{HenRaj13}). Since the work of Iwaniec and \v{S}ver\'ak \cite{IwaSve93}, a rich theory of mappings of finite distortion has been developed (see \cite{AIM09}, \cite{HenKos14}), with
applications to PDE, complex dynamics, inverse problems and non-linear elasticity theory, among other fields. 

The theory extends to $W^{1,1}_{\text{loc}}$-maps with exponentially integrable distortion and also to higher dimensions, where continuity, openness and discreteness of quasiregular maps was proved by Reshetnyak already in the 1960s (see \cite{Res67}). Reshetnyak's theorem has been extended to spatial mappings of finite distortion by several authors (see  \cite{VG76}, \cite{ManVil98}, \cite{KaKoMa01}, \cite{IKO01}, \cite{IwaMar01}, \cite{KKMOZ03}, \cite{OZ08}, \cite{Raj10},  \cite{HenRaj13}). 

Partially motivated by works of Heinonen-Rickman \cite{HeiRic02}, Heinonen-Sullivan \cite{HeiSul02} and Heinonen-Keith \cite{HeiKei11} on BLD- and bi-Lipschitz pa\-ra\-met\-ri\-za\-tions of metric spaces, Kirsilä \cite{Kir16} furthermore extended Reshetnyak's theorem to maps $f:X \to \mathbb{R}^n$, where $X$ is a \emph{generalized $n$-manifold} satisfying assumptions such as Ahlfors $n$-regularity and Poincar\'e inequality.  

In this paper we extend the Iwaniec-\v{S}ver\'ak theorem to maps $f:X \to \R^2$, where $X$ is a \emph{metric surface}, i.e., a metric space homeomorphic to a domain in $\R^2$ with locally finite $2$-dimensional Hausdorff measure. The novelty of our results is that we do not impose any additional conditions on $X$. 

Our research is partially inspired by recent advances on the uniformization problem on metric surfaces (see \cite{BonKle02}, \cite{Raj:17}, \cite{Iko:19}, \cite{MW21}, 
\cite{Mei22}, \cite{NR:21}, \cite{NR22}) and the properties of the associated homeomorphisms, such as quasiconformal maps $f:X \to \R^2$. It is desirable to explore the properties of non-homeomorphic maps on metric surfaces. The aim of our paper is to provide the first results in this direction. 


\subsection{Mappings of finite distortion on metric surfaces}\label{sec:Distortiondef}
A (euclidean) \emph{metric surface} $X$ is a metric space homeomorphic to a domain $U \subset \R^2$ and with locally finite $2$-dimensional Hausdorff measure. 
Below, $\mathcal{H}^2$ will always be the reference measure on $X$. 

Let $X$ and $Y$ be metric surfaces. We want to establish what it means for a map $f\colon X\to Y$ to have finite distortion. We first observe that in the euclidean case every mapping of finite distortion is sense-preserving. This follows from inequality \eqref{eq:distortion_ineq} by applying non-negativity of the Jacobian determinant and integration by parts, a method which is not available in our generality. 
We call $f:X \to Y$ \emph{sense-preserving} if for any domain $\Omega$ compactly contained in $X$ so that $f|_{\partial\Omega}$ is continuous it follows that $\deg(y,f,\Omega)\geq1$ for any $y\in f(\Omega)\setminus f(\partial\Omega)$. Here $\deg$ is the local topological degree of $f$ (see \cite{Ric93}*{I.4}). 

We apply the theory of Sobolev spaces based on \emph{upper gradients} (\cite{HKST:15}). A Borel function $\rho^u:X \to [0,\infty]$ is an \emph{upper gradient} of $f:X \to Y$, if 
\begin{equation}\label{eq:upp}
d_Y(f(x),f(y)) \leq \int_{\gamma} \rho^u \, ds 
\end{equation}
for all $x,y\in X$ and every rectifiable curve $\gamma$ in $X$ joining $x$ and $y$. We say that $f$ belongs to the Sobolev space 
$N^{1,2}_{\text{loc}}(X,Y)$ if $f$ has an upper gradient $\rho^u \in L^2_{\text{loc}}(X)$ and if $d_Y(y,f(\cdot)) \in L^2_{\text{loc}}(X)$ for some $y \in Y$ (see Section \ref{section:Sobolev}). 

It follows from the proof of \cite{EIR22}*{Theorem 1.4} that a sense-preserving map $f \in N^{1,2}_{\text{loc}}(X,\R^2)$ is continuous (see Remark \ref{rem:sense}). Such an $f$ also satisfies Lusin's Condition $(N)$: if $E \subset X$ and $\mathcal{H}^2(E)=0$, then 
$|f(E)|_2=0$ (see Remark \ref{rmk:NN}). The converse implication does not hold (\cite{Raj:17}*{Section 17}). 

In order to define the distortion of $f$, we introduce \emph{lower gradients}: a Borel function $\rho^l:X \to [0,\infty]$ is a \emph{lower gradient} of $f\in\Nloc(X,Y)$, if $\rho^l\leq\rho_f^u$ almost everywhere and
\begin{equation} \label{eq:low} 
\ell(f\circ \gamma) \geq \int_{\gamma} \rho^l \, ds 
\end{equation}
for every rectifiable curve $\gamma$ in $X$ with $f\circ\gamma$ being continuous. 
Our definition is motivated by the observation that the upper gradient inequality \eqref{eq:upp} is equivalent to the reverse inequality of \eqref{eq:low} for $\rho^u$ (see Section \ref{section:Sobolev}). Every $f \in N^{1,2}_{\text{loc}}(X,Y)$ has an essentially unique \emph{minimal weak upper gradient} $\rho_f^u$ (see Section \ref{section:Sobolev}). Similarly, we prove in Section \ref{section:lower-gradient} that every such $f$ has an essentially unique \emph{maximal weak lower gradient} $\rho_f^l$. 

We say that a sense-preserving $f\in\Nloc(X,Y)$ has \emph{finite distortion (along paths) and denote $f \in \operatorname{FDP}(X,Y)$}, if there is a measurable $K\colon X\to[1,\infty)$ such that 
\begin{equation} \label{eq:wait} 
\rho_f^u(x)\leq K(x)\cdot \rho_f^l(x) \quad \text{for almost every } x\in X.
\end{equation}
The \emph{distortion} $K_f$ of $f$ is   
$$K_f(x):=\begin{cases}
\frac{\rho_f^u(x)}{\rho_f^l(x)}, & \text{if } \rho_f^l(x)\neq0,\\
1, & \text{if } \rho_f^l(x)=0.
\end{cases} $$

Our main result is the following extension of the Iwaniec-\v{S}ver\'ak theorem. Here $X$ is any metric surface. 

\begin{thm}\label{thm:main}
Let $f \in \operatorname{FDP}(X,\R^2)$ be non-constant with $K_f\in\Lloc^1(X)$. Then $f$ is open and discrete. 
\end{thm}

Generalizing the euclidean result by Hencl-Koskela (who assumed $W^{1,1}$-regularity, see \cite{HenKos06}), we show that if $f$ is furthermore a homeomorphism, then the inverse is also a Sobolev map. 

\begin{thm}\label{thm:inverse-Sobolev}
     Let $f \in \operatorname{FDP}(X,\R^2)$ be injective with $K_f\in\Lloc^1(X)$. Then $f^{-1}\in\Nloc(f(X),X)$. 
\end{thm}

Examples in \cite{Bal81} ($f_0$ in Proposition \ref{prop:X-not-reciprocal} below, see also  \cite{HenRaj13}) and \cite{HenKos06}*{Example 1.4}, respectively, show that condition $K_f \in \Lloc^1(X)$ is sharp both in Theorem \ref{thm:main} and in Theorem \ref{thm:inverse-Sobolev}, even if $X=\R^2$. 

We show in Section \ref{section:reciprocal} that there are metric surfaces $X$ which do not admit any quasiconformal maps $h:X \to \R^2$ but do admit maps $f:X \to \R^2$ satisfying the assumptions of Theorem \ref{thm:main}. By \cite{MR24}*{Theorem 1.3}, such surfaces do not exist if we require $K_f$ to be bounded instead of integrable. 

Previous approaches to distortion of maps between metric spaces are mostly based on the \emph{analytic definition}: 
We say that a sense-preserving $f\in\Nloc(X,Y)$ has \emph{finite analytic distortion} and denote $f\in\operatorname{FDA}(X,Y)$, if there is a measurable $C\colon X\to[1,\infty)$ such that 
\begin{equation}\label{ineq:analytic-distortion}
\rho^u_f(x)^2 \leq C(x)\cdot J_f(x) \quad \text{for almost every } x\in X,
\end{equation} 
where 
$$
\label{area_limsup}
J_f(x)=\limsup_{r \to 0} \frac{\hm^2_Y(f(\overline{B}(x,r)))}{\pi r^2}. 
$$

Inequality \eqref{ineq:analytic-distortion} is equivalent to \eqref{eq:wait} in the euclidean setting, and also provides a rich theory for homeomorphisms between metric spaces. However, unlike our approach based on lower gradients, the analytic approach is not convenient for treating non-homeomorphic maps between metric surfaces. We nevertheless prove the following in \cite{MR24}. 

\begin{thm}[\cite{MR24}*{Theorem 1.1}]\label{thm:mr24}
    If $f \in \operatorname{FDA}(X,\R^2)$, then $f \in \operatorname{FDP}(X,\R^2)$. Moreover, for every $C(x)$ in \eqref{ineq:analytic-distortion} we have   
        \begin{align*} 
            K_f(x) \leq 4\sqrt{2}\, C(x) \quad \text{for almost every } x\in X.
        \end{align*}
\end{thm} 
Theorem \ref{thm:main} can be applied to prove the converse of Theorem \ref{thm:mr24} assuming $K_f \in \Lloc^1(X,\R^2)$, see \cite{MR24}. Combining Theorems \ref{thm:main}, \ref{thm:inverse-Sobolev} and \ref{thm:mr24} shows that our main results hold under the analytic assumption. 
\begin{corollary}\label{cor:fda}
 Let $f \in \operatorname{FDA}(X,\R^2)$ be non-constant with $C(x) \in\Lloc^1(X)$. Then $f$ is open and discrete. If $f$ is injective, then $f^{-1}\in\Nloc(f(X),X)$.  
\end{corollary}

The definition of a metric surface can be relaxed by requiring $X$ to be homeomorphic to an oriented topological surface $M$ instead of a domain in $\R^2$. Our definitions and results are local and remain valid under the relaxed definition. We state them only for euclidean metric surfaces to simplify the presentation.

This paper is organized as follows. In Section \ref{section:prel} we recall the background on Analysis in metric spaces needed to prove our main results. In Section \ref{section:area_ineq} we prove an area inequality for maps on the rectifiable part of a metric surface which involves lower gradients and may be of independent interest. We prove Theorems \ref{thm:main} and \ref{thm:inverse-Sobolev} in Sections \ref{sec:discr} and \ref{sec:regul}, respectively. 

The proofs are based on three main tools: the coarea inequality for Sobolev functions on metric surfaces by Meier-Ntalampekos \cite{MN23} and Esmayli-Ikonen-Rajala \cite{EIR22}, weakly quasiconformal parametrizations of metric surfaces by Ntalampekos-Romney \cite{NR22}, \cite{NR:21} and Meier-Wenger \cite{MW21}, and the area inequality proved in Section \ref{section:area_ineq}. In addition, to prove Theorem \ref{thm:main} we apply estimates inspired by the value distribution theory of quasiregular mappings (see \cite{Ric93}). 

In Section \ref{section:reciprocal}, we discuss connections between our results and the uniformization problem on metric surfaces, as well as different definitions of mappings with controlled distortion. Finally, in Section \ref{section:lower-gradient} we prove the existence of maximal weak lower gradients.


\section{Preliminaries} \label{section:prel}

\subsection{Basic definitions and notations}

Let $(X,d)$ be a metric space. We denote the \emph{open} and \emph{closed ball} in $X$ of radius $r>0$ centered at a point $x\in X$ by $B(x,r)$ and $\overline{B}(x,r)$, respectively. When $X=\R^2$ we use notation $\D(x,r)$ instead of $B(x,r)$. 

A set $\Omega\subset X$ homeomorphic to the unit disc $\D(0,1)$ is a \emph{Jordan domain} in $X$ if its boundary $\partial\Omega\subset X$ is a \emph{Jordan curve} in $X$, i.e., a subset of $X$ homeomorphic to $\mathbb{S}^1$. The \emph{image} of a curve $\gamma$ in $X$ is indicated by $|\gamma|$ and the 
\emph{length} by $\ell(\gamma)$. 

A curve $\gamma$ is \emph{rectifiable} if $\ell(\gamma)<\infty$ and \emph{locally rectifiable} if each of its compact subcurves is rectifiable. Moreover, a curve $\gamma\colon[a,b]\to X$ is \emph{geodesic} if $\ell(\gamma)=d(\gamma(a),\gamma(b))$. A curve $\gamma\colon[0,\ell(\gamma)]\to X$ is \emph{parametrized by arclength} if $\ell(\gamma|_I)=|I|_1$ for every interval $I\subset[0,\ell(\gamma)]$. Here, $|\cdot|_n$ denotes the \emph{$n$-dimensional Lebesgue measure}.

For $s\geq 0$, we denote the \emph{$s$-dimensional Hausdorff measure} of $A\subset X$ by $\mathcal{H}^s(A)$. The normalizing constant is chosen so that $|V|_n=\mathcal{H}^n(V)$ for open subsets $V$ of $\R^n$.

We equip $X$ with $\mathcal{H}^2$. Let $L^p(X)$ ($L^p_{\text{loc}}(X)$) denote the space of $p$-integrable (locally $p$-integrable) Borel functions from $X$ to $\R\cup\{-\infty,\infty\}$. Here locally $p$-integrable means $p$-integrable on compact subsets. We say that a subdomain $G$ of $X$ is \emph{compactly contained} in $X$ if the closure $\overline{G}$ is compact.



\subsection{Modulus}\label{section:modulus}
Let $X$ be a metric space and $\Gamma$ be a family of curves in $X$. A Borel function $g\colon X \to [0,\infty]$ is \emph{admissible} for $\Gamma$ if $\int_{\gamma}g\, ds\geq 1$ for all locally rectifiable curves $\gamma\in \Gamma$. We define the ($2$-)\emph{modulus} of $\Gamma$ as 
$$\mod \Gamma = \inf_g \int_X g^2 \, d\mathcal H^2,$$
where the infimum is taken over all admissible functions $g$ for $\Gamma$. If there are no admissible functions for $\Gamma$ we set $\mod \Gamma = \infty$. A property is said to hold for \emph{almost every} curve in $\Gamma$ if it holds for every curve in $\Gamma\setminus\Gamma_0$ for some family $\Gamma_0\subset \Gamma$ with $\mod(\Gamma_0)=0$. In the definition of $\mod(\Gamma)$, the infimum can equivalently be taken over all \emph{weakly admissible} functions, i.e.,\ Borel functions $g\colon X\to [0,\infty]$ such that $\int_\gamma g\geq 1$ for almost every locally rectifiable curve $\gamma\in\Gamma$.

\subsection{Metric Sobolev spaces}\label{section:Sobolev}
Let $f\colon X\to Y$ be a map between metric spaces. A Borel function $\rho^u \colon X\to [0,\infty]$ is an \textit{upper gradient} of $f$ if 
\begin{align}\label{ineq:upper_gradient}
    d_Y(f(x),f(y)) \leq \int_{\gamma} \rho^u \, ds
\end{align}
for all $x,y\in X$ and every rectifiable curve $\gamma$ in $X$ joining $x$ and $y$. If the \textit{upper gradient inequality} \eqref{ineq:upper_gradient} holds for almost every rectifiable curve $\gamma$ in $X$ joining $x$ and $y$ we call $\rho^u$ \emph{weak upper gradient} of $f$. 

The Sobolev space $N^{1,2}(X,Y)$ is the space of Borel maps $f \colon X \to Y$ with upper gradient $\rho^u \in L^2(X)$ such that $x \mapsto d_Y(y,f(x))$ is in $L^2(X)$ for some and thus any $y \in Y$. The space $\Nloc(X, Y)$ is defined in the obvious manner. 

Each $f\in \Nloc(X,Y)$ has a \textit{minimal} weak upper gradient $\rho_f^u$, i.e.,\ for any other weak upper gradient $\rho^u$ we have $\rho_f^u\leq \rho^u$ almost everywhere. Moreover, $\rho^u_f$ is unique up to a set of measure zero. See monograph \cite{HKST:15} for more background on metric Sobolev spaces.

We apply a notion of ``minimal stretching'' which compliments the ``maximal stretching'' represented by upper gradients. To motivate the definition, notice that for continuous maps $f \in \Nloc(X,Y)$ the upper gradient inequality \eqref{ineq:upper_gradient} is equivalent to 
\[
 \ell(f\circ\gamma)\leq\int_{\gamma}\rho^u \, ds  
\qquad
\] 
for almost every rectifiable curve $\gamma$ in $X$.
 We call a Borel function $\rho^l\colon X\to[0,\infty]$ a \emph{lower gradient} of $f\in\Nloc(X,Y)$, if $\rho^l\leq\rho_f^u$ almost everywhere and 
\begin{align}\label{ineq:lower_gradient}
    \ell(f\circ\gamma)\geq\int_{\gamma}\rho^l \, ds 
\end{align}
for every rectifiable curve $\gamma$ in $X$ with $f\circ\gamma$ being continuous. 
If the \textit{lower gradient inequality} \eqref{ineq:lower_gradient} holds for almost every rectifiable $\gamma$, we call $\rho^l$ \emph{weak lower gradient} of $f$. Note that $0$ is always a lower gradient. 

Each $f\in \Nloc(X,Y)$ has a \emph{maximal weak lower gradient} $\rho_f^l$, i.e., for any other weak lower gradient $\rho^l$ we have $\rho_f^l\geq\rho^l$ almost everywhere. Moreover, $\rho^l_f$ is unique up to a set of measure zero. The proof is analogous to the existence of minimal weak upper gradients, see \cite{HKST:15}*{Theorem 6.3.20}. For completeness, we provide a proof in Section \ref{section:lower-gradient}.


\subsection{Coarea inequality on metric surfaces}
We  state the following coarea inequality for Lipschitz functions, which is a consequence of \cite{EH21}*{Theorem 1.1} (see \cite{EIR22}*{Section 5}). Here, $\lip(u)$ denotes the pointwise Lipschitz constant of a Lipschitz function $u\colon X\to \R$, defined by $$\lip(u)(x)=\limsup_{x \neq y\to x}\frac{|u(y)-u(x)|}{d(x,y)}. $$
\begin{thm}[Lipschitz coarea inequality]\label{thm:coarea-Lipschitz}
    Let $X$ be a metric space and $u\colon X\to \R$ a Lipschitz function. Then 
    \begin{align*}
        \int\displaylimits^*_\R \int_{u^{-1}(t)} g\, d\mathcal H^1dt \leq \frac{4}{\pi} \int_X g\cdot \lip(u)   \, d\mathcal H^2  
    \end{align*} 
for every Borel measurable $g:X \to [0,\infty]$. 
\end{thm} 
Here $\int^*$ denotes the upper integral, which is equal to Lebesgue integral for measurable functions.
An important tool throughout this work will be the following coarea inequality for continuous Sobolev functions on metric surfaces.

\begin{thm}[Sobolev coarea inequality, \cite{MN23}*{Theorem 1.6}]\label{thm:coarea-Sobolev}
    Let $X$ be a metric surface and $v\colon X\to \R$ be a continuous function in $\Nloc(X)$.  
    \begin{enumerate}[label=\normalfont(\arabic*)]
        \item If $\mathcal A_v$ denotes the union of all non-degenerate components of the level sets $v^{-1}(t)$, $t\in \R$, of $v$, then $\mathcal A_v$ is a Borel set. \label{ca:i}
        \item For every Borel function $g\colon X\to [0,\infty]$ we have
        $$\int\displaylimits^* \int_{v^{-1}(t)\cap \mathcal A_v}g\, d\mathcal H^1\, dt \leq \frac{4}{\pi} \int g\cdot\rho_v^u\, d\mathcal H^2.$$\label{ca:ii}
    \end{enumerate}
\end{thm}

Theorem \ref{thm:coarea-Sobolev} generalizes the coarea inequality for monotone Sobolev functions established in \cite{EIR22}. 
Here $v\colon X\to\R$ is called a \emph{weakly monotone function} if for every open $\Omega$ compactly contained in $X$
$$\sup_\Omega v\leq\sup_{\partial\Omega} v<\infty\quad\text{ and }\quad\inf_\Omega v\geq\inf_{\partial\Omega} v>-\infty.$$
A continuous weakly monotone function is \emph{monotone}.

\begin{rmk} \label{rem:sense}
    In the proof of \cite{EIR22}*{Theorem 1.4} the coarea inequality for monotone Sobolev functions is used to show that every weakly monotone function $v\in\Nloc(X,\R)$ is continuous and hence monotone. Continuity of a sense-preserving map $f\in\Nloc(X,\R^2)$ now follows by applying the exact same proof strategy while replacing weak monotonicity with sense-preservation and the coarea inequality for monotone Sobolev maps with Theorem \ref{thm:coarea-Sobolev}. 
\end{rmk}

\subsection{Metric differentiability}
Let $(Y,d)$ be a complete metric space and  $U\subset\R^n$, $n \geq 1$, a domain. We say that $h\colon U\to Y$ is \emph{approximately metrically differentiable} at $z\in U$ if there exists a seminorm $N_z$ on $\R^2$ for which
    $$\ap\lim_{y\to z}\frac{d(h(y),h(z))-N_z(y-z)}{|y-z|}=0.$$
Here, $\ap\lim$ denotes the approximate limit (see \cite{EG92}*{Section 1.7.2}). If such a seminorm exists, it is unique and is called \emph{approximate metric derivative} of $h$ at $z$, denoted $\apmd h_z$. 
The following result follows from \cite{LWintrinsic}*{Lemma 3.1}. 

\begin{lemma} \label{lemma:length_cov}
Let $X$ and $Y$ be metric surfaces and $f\in\Nloc(X,Y)$. Almost every curve $\gamma\colon [a,b]\to X$ parametrized by arclength satisfies 
$$
\int_{f \circ \gamma} g \, ds = 
\int_a^b g(f(\gamma(t))) \cdot \apmd(f\circ\gamma)_t \, dt 
$$ 
for all Borel measurable $g:Y \to [0,\infty]$. 
\end{lemma}

Lemma \ref{lemma:length_cov} leads to the following properties of upper and lower gradients (see \cite{HKST:15}*{Proposition 6.3.3} for a proof involving upper gradients). 

\begin{corollary} \label{cor:gchange}
Let $X$ and $Y$ be metric surfaces and $f\in\Nloc(X,Y)$. Almost every curve $\gamma\colon [a,b]\to X$ parametrized by arclength satisfies the following properties. 
\begin{enumerate} 
\item $f$ is absolutely continuous on $\gamma$, 
\item $\rho_f^l(\gamma(t)) \leq \apmd(f\circ\gamma)_t \leq \rho_f^u(\gamma(t))$ for almost every $a<t<b$, 
\item if $g:Y \to [0,\infty]$ is Borel measurable, then 
$$
\int_\gamma \rho_f^l \cdot (g\circ f) \, ds \leq \int_{f \circ \gamma} g \, ds \leq \int_\gamma \rho_f^u \cdot (g\circ f) \, ds. 
$$
\end{enumerate}
\end{corollary}


\subsection{Area formula on euclidean domains}

Suppose $U \subset \R^2$ is a domain and $h\in\Nloc(U,Y)$. Then $U$ can be covered up to a set of measure zero by countably many disjoint measurable sets $G_j$, $j\in \N$, such that $h|_{G_j}$ is Lipschitz. In particular, outside a set of measure zero $G_0\subset U$, $h$ satisfies Lusin's condition (N) (see \cite{HKST:15}*{Theorem 8.1.49}). 

By \cite{LWarea}*{Proposition 4.3}, every $h\in \Nloc(U,Y)$ is approximately metrically differentiable at a.e.\ $z\in U$. The following area formula follows from \cite{Kar07}*{Theorem 3.2}. Here, the \emph{Jacobian} $J(N_z)$ of a seminorm $N_z$ on $\R^2$ is zero if $N_z$ is not a norm and $J(N_z)={\pi}/{|\{y\in\R^2: N_z(y)\leq1\}|_2}$ otherwise.

\begin{thm}[Area formula]\label{thm:area-formula}
    If $h\in\Nloc(U,Y)$, then there exists $G_0\subset U$ with $\mathcal H^2(G_0)=0$ such that for every measurable set $A\subset U\setminus G_0$ we have
    \begin{align}\label{eq:area-formula}
       \int_A J(\ap\md h_z) \,d\mathcal H^2=\int_Y N(y,h,A)\,d\mathcal H^2.
    \end{align}
\end{thm}

Here, $N(y,h,A)$ denotes the \emph{multiplicity} of $y\in Y$ with respect to $h$ in $A$:  
\begin{equation} \label{eq:multip}
N(y,h,A):=\#\{z\in A: h(z) = y\}. 
\end{equation}

\subsection{Weakly quasiconformal parametrizations}
A map $h\colon X\to Y$ between metric surfaces is \textit{cell-like} if the preimage of each point is a continuum that is contractible in each of its open neighborhoods. A continuous, surjective, proper and cell-like map $h\colon X\to Y$ is \emph{weakly $C$-quasiconformal} if
$$\mod \Gamma\leq C \mod h(\Gamma)$$
holds for every family of curves $\Gamma$ in $X$. It follows from \cite{Wil:12}*{Theorem 1.1} that every weakly quasiconformal map $h\colon X\to Y$ is contained in $\Nloc(X,Y)$.

It was shown in \cite{NR22} that any metric surface admits a weakly quasiconformal parametrization, see also \cite{NR:21}, \cite{MW21}, \cite{Mei22}.

\begin{thm}[\cite{NR22}*{Theorem 1.2}]\label{theorem:wqc}
    Let $X$ be any metric surface. There is a weakly $(4/\pi)$-quasiconformal  $u\colon U \to X$, where $U\subset \R^2$ is a domain. 
\end{thm}

\begin{rmk} \label{rmk:NN}
Condition (N) for sense-preserving maps $f \in \Nloc(X,\R^2)$ can be proved using the area formula and Theorem \ref{theorem:wqc} as follows: suppose $E \subset X$ and $\hm^2(E)=0$, and let $u\colon U \to X$ be a (sense-preserving) weakly $(4/\pi)$-quasiconformal parametrization of $X$ provided by Theorem \ref{theorem:wqc}. Define $h\colon U \to\R^2$ by $h:=f\circ u$. Then $u\in\Nloc(U,X)$ and $h\in\Nloc(U,\R^2)$, see \cite{MR24}*{Theorem 2.5}. 

By Theorem \ref{thm:area-formula} there exists $G_0\subset U$ with $|G_0|_2=0$ and such that \eqref{eq:area-formula} holds for $u$ and $h$ and every measurable set $A\subset U\setminus G_0$. We set $X_0:=u(G_0)$.
Now $h$ is sense-preserving and thus monotone. Therefore, $h$ satisfies Condition (N) by \cite{MalMar95}. In particular, with the above notation, 
$$
|f(E)|_2 \leq \int_{u^{-1}(E)} J(\apmd h_z) \, dz. 
$$
On the other hand, applying Theorem \ref{thm:area-formula} to $u$ shows that 
$$
\int_{u^{-1}(E)} J(\apmd u_z) \, dz \leq \hm^2(E)=0, 
$$
and so $J(\apmd u_z)=0$ almost everywhere in $u^{-1}(E)$. Since $u$ is weakly quasiconformal, it moreover follows that $\apmd u_z=0$. Then, by Lemmas \ref{lemma:min-max-stretch} and \ref{lem:compo} below, $J(\apmd h_z)=0$ 
almost everywhere in $u^{-1}(E)$ as well. We conclude that $|f(E)|_2=0$.   
\end{rmk}


\subsection{Distortion of Sobolev maps}
Let $U \subset \R^2$ be a domain. 
We define the \emph{maximal and minimal stretches} of $h\in\Nloc(U,Y)$ at points of approximate differentiability by 
\begin{eqnarray*} 
L_h(z)=\max\{\ap\md h_z(v):|v|=1\}, \quad  
l_h(z)=\min\{\ap\md h_z(v):|v|=1\}. 
\end{eqnarray*}
Recall that maps $h \in \Nloc(U,Y)$ are approximately differentiable almost everywhere. 

\begin{lemma} \label{lemma:min-max-stretch}
Let $h \in \Nloc(U,Y)$. Then $L_h$ and $l_h$ are representatives of the minimal weak upper gradient and the maximal weak lower gradient of $h$, respectively. Moreover, 
\begin{equation}\label{eq:disto}
2^{-1} L_h(z)l_h(z) \leq J(\ap\md h_z) \leq 
2 L_h(z)l_h(z)  
\end{equation} 
at points of approximate differentiability.  
\end{lemma}
\begin{proof}
The first claim concerning upper gradients is \cite{MN23}*{Lemma 2.14}. A slight modification of the proof gives the claim concerning lower gradients. 

Towards \eqref{eq:disto}, we may assume that $\ap\md h_z$ is a norm. Then the unit ball $B_z$ of $\ap\md h_z(v)$ contains a unique ellipse of maximal area $E_z$, called the \emph{John ellipse} of $B_z$, which satisfies 
\begin{equation}\label{eq:John}
E_z \subset B_z \subset \sqrt{2}E_z,
\end{equation}
see \cite{Bal97}*{Theorem 3.1}. Let $N_z$ be the norm whose unit ball is $E_z$, and  
\[
M_z=\max\{N_z(v):|v|=1\}, \quad m_z=\min\{N_z(v):|v|=1\}. 
\]
Then $J(N_z)=\pi/|E_z|_2=M_zm_z$, and \eqref{eq:John} gives 
\[
L_h(z)l_h(z) \leq M_zm_z =J(N_z)= 2\pi/|\sqrt{2}E_z|_2 \leq 2 \pi/ |B_z|_2=2J(\ap\md h_z). 
\] 
On the other hand, \eqref{eq:John} also gives  
\[
J(\ap\md h_z) \leq J(N_z)=M_zm_z\leq 2L_h(z)l_h(z). 
\] 
The proof is complete. 
\end{proof}

We will apply distortion estimates on composed mappings. 

\begin{lemma}\label{lem:compo}
Let $X$ and $Y$ be metric surfaces and $U \subset \R^2$ a domain, $u:U \to X$ weakly quasiconformal, and $f \in \Nloc(X,Y)$. Then $$
l_{f \circ u}(z)\geq \rho_f^l(u(z))\cdot l_u(z) 
\quad \text{and} \quad 
L_{f \circ u}(z)\leq \rho_f^u(u(z))\cdot L_u(z) 
$$
for almost every $z \in U$. 
\end{lemma}
\begin{proof}
Let $\Gamma_0$ be the family of paths $\gamma$ in $U$ so that $l_u$ does not satisfy the lower gradient inequality \eqref{ineq:lower_gradient} for $u$ on some subcurve of $\gamma$ or $\rho_f^l$ does not satisfy the lower gradient inequality for $f$ on some subcurve of $u\circ \gamma$. Then, since $u$ is weakly quasiconformal and $l_u$, $\rho_f^l$ are weak lower gradients (Lemma \ref{lemma:min-max-stretch}), we conclude that $\mod(\Gamma_0)=0$. Applying Corollary \ref{cor:gchange}, we have 
$$
\ell(f \circ u \circ \gamma) \geq \int_{u \circ \gamma} \rho_f^l \, ds \geq \int_\gamma (\rho_f^l\circ u)\cdot l_u \, ds
$$ 
for every $\gamma \notin \Gamma_0$ parametrized by arclength. We conclude that 
$(\rho_f^l\circ u)\cdot l_u$ is a weak lower gradient of $f \circ u$. But $l_{f \circ u}$ is a maximal weak lower gradient of $f \circ u$ by Lemma \ref{lemma:min-max-stretch}. The first inequality follows. The second inequality is proved in a similar way. 
\end{proof}


\section{Area inequality on Metric surfaces} \label{section:area_ineq}
Let $X$ and $Y$ be metric surfaces. In this section we establish Theorem \ref{thm:area-ineq}, an area inequality for Sobolev maps in $\Nloc(X,Y)$ on measurable subsets of the rectifiable part of $X$. We apply Theorem \ref{thm:area-ineq} in Sections \ref{sec:discr} and \ref{sec:regul} below to prove our main results, Theorems \ref{thm:main} and \ref{thm:inverse-Sobolev}. 

As in Remark \ref{rmk:NN}, let $u\colon U \to X$ be a weakly $(4/\pi)$-quasiconformal parametrization of $X$ provided by Theorem \ref{theorem:wqc}, and $h\colon U \to Y$,  $h:=f\circ u$. Then $u\in\Nloc(U,X)$ and $h\in\Nloc(U,Y)$. By Theorem \ref{thm:area-formula}, there exists $G_0\subset U$ with $|G_0|_2=0$ and such that \eqref{eq:area-formula} holds for both $u$ and $h$ and every measurable set $A\subset U\setminus G_0$. We set $X_0:=u(G_0)$.
\begin{thm}[Area inequality]\label{thm:area-ineq}
    If $g\colon Y\to [0,\infty]$ and $E\subset X\setminus X_0$ are Borel measurable, then
    \begin{align*}
        \int_Eg(f(x))\cdot\rho_f^u(x)\rho_f^l(x)\,d\hm^2\leq 4\sqrt{2}\int_{Y}g(y)\cdot N(y,f,E)\,dy.
    \end{align*}
    If in addition, the map $f$ satisfies Lusin's condition (N), then
    \begin{align*}
        \int_Eg(f(x))\cdot\rho_f^u(x)\rho_f^l(x)\,d\hm^2\geq \frac{1}{4\sqrt{2}}\int_{Y}g(y)\cdot N(y,f,E)\,dy. 
    \end{align*}
\end{thm}

In order to establish Theorem \ref{thm:area-ineq}, we make use of the following proposition which can be seen as a counterpart to Lemma \ref{lem:compo}. 

\begin{prop} \label{prop:mildest}
Let $f$, $u$ and $h=f\circ u$ be as above. Then  
\begin{equation}\label{eq:inversecompo}
\rho_f^u  (u(z)) \cdot l_u(z) \leq L_h(z) \quad 
\text{and} \quad l_h(z) \leq \rho_f^l(u(z)) \cdot L_u(z) 
\end{equation}
for almost every $z \in U\setminus G_0$. 
\end{prop}

\begin{proof} 
Fix Borel representatives of the maps $z\mapsto\apmd u_z$ and $z\mapsto\apmd h_z$. Towards the first inequality in \eqref{eq:inversecompo}, we denote 
$$
G_0'=G_0 \cup \{z \in U: \, l_u(z)=0\}, 
$$
and notice that it suffices to prove the inequality for almost every $z \in U \setminus G_0'$. By \cite{LWarea}*{Proposition 4.3}, there are pairwise disjoint Borel sets $K_i \subset U \setminus G_0'$, $i \in \mathbb{N}$, 
so that 
\begin{equation}\label{downdown}
|U \setminus (G_0' \cup (\cup_i K_i))|_2=0
\end{equation} 
and so that for every $i \in \mathbb{N}$ we have 
\begin{itemize}
\item[(i)] $\apmd u_z$ and $\apmd h_z$ exist for every $z \in K_i$ and 
\item[(ii)] for every $\varepsilon>0$ there is $r_i(\varepsilon)>0$ so that 
\begin{eqnarray*}
& & |d_X(u(z+v),u(z+w))-\apmd u_z(v-w)| \leq \varepsilon |v-w| \quad \text{and } \\ 
& & |d_Y(h(z+v),h(z+w)) - \apmd h_z(v-w)| \leq \varepsilon |v-w| 
\end{eqnarray*}
for every $z \in K_i$ and all $v,w \in \mathbb{R}^2$ with $|v|,|w| \leq r_i(\varepsilon)$ and such that $z+v,z+w \in K_i$. 
\end{itemize}

We will show that if $i \in \mathbb{N}$ then almost every curve $\gamma$ in $X$ parametrized by arclength has the following property: almost every $t \in \gamma^{-1}(u(K_i))$ satisfies 
\begin{equation}\label{limppu} 
\apmd(f\circ\gamma)_t \leq \frac{L_h(z)}{l_u(z)} 
\quad \text{for all } z \in u^{-1}(\gamma(t)) \cap K_i.
\end{equation}
  
We show how to conclude the first inequality in \eqref{eq:inversecompo} from \eqref{limppu}. By Lemma \ref{lemma:length_cov}, Corollary \ref{cor:gchange} and \eqref{limppu}, $\rho\colon X\to[0,\infty]$ is a weak upper gradient of $f$, where 
$\rho(x)=\rho^u_f(x)$ for $x \in X \setminus u(K_i)$ and 
$$\rho(x)=\inf_{z \in K_i, \, u(z)=x} \frac{L_h(z)}{l_u(z)} 
$$
when $x \in u(K_i)$. By the definition of minimal weak upper gradients, we then have that 
\begin{equation} \label{eq:jisma}
\rho^u_f(x) \leq \rho(x) \quad \text{for almost every }x \in u(K_i). 
\end{equation} 
Since $K_i \subset U \setminus G'_0$, we have $l_u>0$ and thus $J(\ap\md u_z)>0$ in $K_i$. Combining \eqref{eq:jisma} with the Area formula (Theorem \ref{thm:area-formula}) for $u$ now yields  
$$
\rho_f^u  (u(z)) \cdot l_u(z) \leq L_h(z) 
$$ 
for almost every $z \in K_i$. The first inequality in \eqref{eq:inversecompo} follows from \eqref{downdown}.  

We now prove \eqref{limppu}. Denote by $\widehat{X}\subset X$ the set of points $x$ for which $N(x,u,U)=1$. By \cite{NR:21}*{Remark 7.2}, $\mathcal{H}^2(X \setminus \widehat{X})=0$. In particular, almost every rectifiable curve $\gamma:[0,\ell(\gamma)] \to X$ parametrized by arclength satisfies $\gamma(t)\in \widehat{X}$ for $\mathcal{H}^1$-almost every $0<t<\ell(\gamma)$. 

We fix such a $\gamma$ and a density point $t_0 \in \gamma^{-1}(u(K_i)\cap \widehat{X})=:T$ of $T$. By Corollary \ref{cor:gchange}, we may moreover assume that 
$f\circ \gamma$ is approximately metrically differentiable at $t_0$. It suffices to show that \eqref{limppu} holds for $t_0$ and the unique $z_0=u^{-1}(\gamma(t_0)) \in K_i$. 

Fix a sequence $(t_j)$ of points in $T$ converging to $t$. Then $x_j:=\gamma(t_j) \to \gamma(t_0)=:x_0$. Moreover, since $x_0 \in \widehat{X}$, we have 
$z_j:=u^{-1}(x_j) \to z_0$. We are now in position to apply Property (ii) above. Denoting $y_j=f(x_j)$ for $j=0,1,\ldots$, (ii) and triangle inequality yield 
\begin{eqnarray*}
\frac{d_X(x_j,x_0)}{|z_j-z_0|} &\geq& \apmd u_{z_0}\Big(\frac{z_j-z_0}{|z_j-z_0|}\Big) -o(|z_j-z_0|) 
\geq l_u(z_0)-o(|z_j-z_0|), \\ 
\frac{d_Y(y_j,y_0)}{|z_j-z_0|} &\leq& \apmd h_{z_0}\Big(\frac{z_j-z_0}{|z_j-z_0|}\Big) +o(|z_j-z_0|) \leq L_h(z_0)+o(|z_j-z_0|). 
\end{eqnarray*}
Combining the inequalities, we have 
\begin{equation}\label{nakki}  
\frac{d_Y(y_j,y_0)}{d_X(x_j,x_0)}= 
\frac{d_Y(y_j,y_0)\cdot |z_j-z_0|}{|z_j-z_0|\cdot d_X(x_j,x_0)} 
\leq \frac{L_h(z_0)}{l_u(z_0)} + o(|z_j-z_0|). 
\end{equation}
Since $\gamma$ is parametrized by arclength, \eqref{nakki} gives \eqref{limppu}. The first inequality in \eqref{eq:inversecompo} follows. The second inequality follows in a similar way, namely showing that instead of \eqref{limppu} we have   
$$
\apmd(f\circ\gamma)_t \geq \frac{l_h(z)}{L_u(z)} 
$$
outside suitable exceptional sets. We leave the details to the reader. 
\end{proof}

\begin{proof}[Proof of Theorem \ref{thm:area-ineq}]
    We may approximate $g$ with simple functions and replace $E$ with appropriate subsets to see that it suffices to show the claim for $g \equiv 1$. We set $E'=E\cap \widehat{X}$, where $\widehat{X}$ is as in the proof of Proposition \ref{prop:mildest}, and obtain
    \begin{align}\label{eq:mult-precomp-wqc}
        N(y,h,u^{-1}(E'))=\sum_{x\in f^{-1}(y)}
        N(x,u,u^{-1}(E'))= N(y,f,E')
    \end{align}
    for every $y\in f(E')$. 
    
     The area formula (Theorem \ref{thm:area-formula}) implies
    \begin{align*}
        \int_E\rho_f^u(x)\rho_f^l(x)\,d\hm^2&=
        \int_{E'}\rho_f^u(x)\rho_f^l(x)
        N(x,u,u^{-1}(E'))\,d\hm^2\\
        &=\int_{u^{-1}(E')}\rho_f^u(u(z))\rho_f^l (u(z))J(\ap\md u_z)\,dz.
    \end{align*}
 By Lemma \ref{lemma:min-max-stretch}, $J(\ap\md u_z)\leq 2 L_u(z)\cdot l_u(z)$ for almost every $z\in u^{-1}(E')$. Moreover, it follows from the proof of Theorem \ref{theorem:wqc} given in \cite{NR22} that we can choose $u$ so that the John ellipse of $\ap\md u_z$ (see \eqref{eq:John}) is a disk. Then $L_u(z)\leq \sqrt{2} l_u(z)$, which leads to  
$$
J(\ap\md u_z) \leq 2 L_u(z)\cdot l_u(z)
\leq 2\sqrt{2}\cdot l_u(z)^2 \quad \text{for almost every } z\in u^{-1}(E'). 
$$ 
Combining with Lemma \ref{lem:compo} and Proposition \ref{prop:mildest}, we conclude that 
    \begin{align*}
 \int_E\rho_f^u(x)\rho_f^l(x)\,d\hm^2
 \leq 2\sqrt{2} \int_{u^{-1}(E')} L_h(z) l_h(z)\,dz.
    \end{align*}
   Applying Lemma \ref{lemma:min-max-stretch} and the area formula (Theorem \ref{thm:area-formula}) to $h$, we finally obtain
    \begin{align*}
        \int_E\rho_f^u(x)\rho_f^l(x)\,d\hm^2& \leq 4\sqrt{2}\int_{u^{-1}(E')}J(\ap\md h_z)\,dz\\
        &= 4\sqrt{2}\int_{f(E')}N(y,h,u^{-1}(E'))\,dy.
    \end{align*}
    The theorem follows by combining with \eqref{eq:mult-precomp-wqc}. 
    
    For the second statement we note that $f$ satisfying Lusin's condition (N) implies $\hm^2(f(E\setminus E'))=0$ as, by \cite{NR:21}*{Remark 7.2}, $\hm^2(E\setminus E')=0$. The rest of the proof is analogous to the arguments above.
\end{proof}


\section{Openness and discreteness} \label{sec:discr}
Throughout this section let $f$ be as in Theorem \ref{thm:main}, i.e., \ $f\in\Nloc(X,\R^2)$ is non-constant, sense-preserving and satisfies $K_f\in\Lloc^1(X)$. Recall that $f$ is continuous by Remark \ref{rem:sense}. 

A map $f\colon X \to\R^2$ is \textit{light} if $f^{-1}(y)$ is totally disconnected for every $y\in \R^2$. It is well-known that if $f$ is continuous, sense-preserving and light, then $f$ is open and discrete \cite{TY62}, \cite{Ric93}*{Lemma VI.5.6}. Thus, in order to prove Theorem \ref{thm:main} it suffices to show that $f$ is in fact light. The proof of this fact relies on the following two propositions involving estimates on the multiplicity of $f$ (recall notation $N(y,h,A)$ for multiplicity in \eqref{eq:multip}).

\begin{prop}\label{prop:discr-open-nbhd}
    Suppose that there are  $s,r_0>0$ and $C>0$ such that
    \begin{align}\label{ineq:multiplicity}
       \int_0^{2\pi} N(f(x_0)+re^{i\theta},f,B(x_0,s)) \, d\theta \leq C\log \frac{1}{r}
    \end{align}
    for all $r<r_0$. Then the $x_0$-component of $f^{-1}(f(x_0))$ either is $\{x_0\}$ or contains an open neighborhood of $x_0$.  
\end{prop}

Recall that $X$ is homeomorphic to a planar domain. In particular, for every $x_0 \in X$ there is $s>0$ so that $\overline{B}(x_0,2s)$ is a compact subset of $X$. 

\begin{prop}\label{prop:ineq-multiplicity}
    Let $x_0 \in X$ and $s>0$ so that $\overline{B}(x_0,2s) \subset X$ is compact. Then Condition \eqref{ineq:multiplicity} holds with some $r_0,C>0$. 
\end{prop}

Theorem \ref{thm:main} follows by combining Propositions \ref{prop:discr-open-nbhd} and \ref{prop:ineq-multiplicity}: since $f$ is not constant, for every $y_0 \in f(X)$ every component $F$ of $f^{-1}(y_0)$ contains a point $x_0 \in X$ which is a boundary point of $F$. Combining Propositions \ref{prop:discr-open-nbhd} and \ref{prop:ineq-multiplicity}, we see that $F=\{x_0\}$. We conclude that $f$ is light and therefore open and discrete.

\subsection{Proof of Proposition \ref{prop:discr-open-nbhd}}
Let $f\colon X\to\R^2$ be a map of finite distortion and $\Gamma$ a curve family in $X$. We define the weighted modulus
$$\mod_{K^{-1}}\Gamma=\inf_{g}\int_{X}\frac{g(x)^2}{K_f(x)}\,d\hm^2,$$
where the infimum is taken over all weakly admissible functions $g$ for $\Gamma$.

Let $u\colon U\to X$ be a weakly $(4/\pi)$-quasiconformal parametrization of $X$ as in Theorem \ref{theorem:wqc}. Let $G_0\subset U$ and $X_0=u(G_0)\subset X$ be as in the paragraph preceding Theorem \ref{thm:area-ineq}. Recall that $|G_0|_2=0$. We set $X':=X\setminus X_0$.

\begin{lemma}\label{lem:upper-bound-Kmodulus}
    Let $\Gamma'$ be a family of curves in $\Omega \subset X$ with $\mathcal{H}^1(|\gamma|\cap X_0)=0$ for every $\gamma\in\Gamma'$. 
    Then
    $$\mod_{K^{-1}}\Gamma'\leq 4\sqrt{2} \int_{\R^2}g(y)^2
    N(y,f,\Omega)\,dy,$$
    whenever $g$ is admissible for $\Gamma=f(\Gamma')$.
\end{lemma}

\begin{proof}
    Fix an admissible $g$ for $\Gamma$, and let $g'\colon X\to\R$, 
    \[
    g'(x):=g(f(x))\cdot\rho_f^u(x)\cdot\chi_{\Omega \cap X'}(x). 
    \]
    Here, $\chi_E$ denotes the indicator function on a set $E\subset X$, i.e.,\ $\chi_E(x)=1$ if $x\in E$ and $\chi_E(x)=0$ else. For almost every $\gamma\in\Gamma'$ we have that $f$ is absolutely continuous on $\gamma$,  $\mathcal{H}^1(|\gamma|\cap X_0)=0$, and 
    \begin{align*}
        \int_{\gamma}g'\,ds=\int_{\gamma}(g\circ f)\cdot\rho_f^u\,ds\geq\int_{f\circ\gamma}g\,ds,
    \end{align*}
    see Corollary \ref{cor:gchange}.  Since $g$ is admissible for $\Gamma=f(\Gamma')$, it follows that $g'$ is weakly admissible for $\Gamma'$. Moreover,

    \begin{align*}
        \mod_{K^{-1}}\Gamma'&\leq\int_{X}\frac{g'(x)^2}{K_f(x)}\,d\hm^2=\int_{\Omega \cap X'}g(f(x))^2\cdot \rho_f^u(x)\rho_f^l(x)\,d\hm^2\\
        &\leq 4\sqrt{2}\int_{\R^2}g(y)^2\cdot N(y,f,\Omega) \,dy,
    \end{align*}
    where the last inequality follows from the area inequality, Theorem \ref{thm:area-ineq}. 
\end{proof}

\begin{lemma}\label{lem:level-set-curves} 
    Let $\varphi\in\Nloc(X,\R)$, and consider $E\subset\R$ with $|E|_1>0$ and so that each level set $\varphi^{-1}(t)$, $t \in E$, contains a non-degenerate continuum $\eta_t$. Then 
    $\mathcal{H}^1(\eta_t\cap X_0)=0$ for almost every $t\in E$.
\end{lemma}

\begin{proof}
Note that $\widehat{\varphi}=\varphi\circ u$ is in $\Nloc(U,\R)$. For every $t\in E$, let $\widehat{\eta_t}=u^{-1}(\eta_t)$.
Then, since $u$ is continuous and proper, $\widehat{\eta_t}$ is a non-degenerate continuum for every $t \in E$. Moreover, the coarea inequality for Sobolev functions (Theorem \ref{thm:coarea-Sobolev}) shows that $\mathcal{H}^1(\widehat{\eta_t})<\infty$ for almost every $t \in E$. For every such $t$, there is a surjective two-to-one $1$-Lipschitz curve 
\[
\widehat{\gamma_t}:[0,2\mathcal{H}^1(\widehat{\eta_t})] \to \widehat{\eta_t}, 
\]
cf. \cite{RR19}*{Proposition 5.1}. 
Let $\widehat{\Gamma}$ be the family of the curves $\widehat{\gamma_t}$, and let $g\colon U \to[0,\infty]$ be admissible for $\widehat{\Gamma}$. We apply the coarea inequality for Sobolev functions (Theorem \ref{thm:coarea-Sobolev}) and Hölder's inequality to obtain
\begin{align*}
    |E|_1&\leq\int\displaylimits^*_E\int_{\widehat{\gamma_t}}g\,ds\,dt
    \leq 2 \int\displaylimits^*_E\int_{\widehat{\eta_t}}g\,d\mathcal{H}^1\,dt
    \leq \frac{8}{\pi}\int_{\widehat{\varphi}^{-1}(E)}g\cdot\rho_{\widehat{\varphi}}^u\,d\hm^2\\
    &\leq\frac{8}{\pi}\left(\int_{\widehat{\varphi}^{-1}(E)}g^2\,d\hm^2\right)^{1/2}\left(\int_{\widehat{\varphi}^{-1}(E)}(\rho_{\widehat{\varphi}}^u)^2\,d\hm^2\right)^{1/2}.
\end{align*}
Since $\rho_{\widehat{\varphi}}^u\in\Lloc^2(U)$ and $|E|_1>0$ it follows that $\mod(\widehat{\Gamma})>0$. 
As a Sobolev function, $u$ is therefore absolutely continuous along $\widehat{\gamma_t}$ for almost every $t\in E$, see e.g.\ \cite{HKST:15}*{Lemma 6.3.1}. Moreover, for almost every $t\in E$ we have that $\mathcal{H}^1(\widehat{\eta_t}\cap G_0)=0$, since $|G_0|_2=0$. Combining these two facts shows that $\mathcal{H}^1(\eta_t\cap X_0)=0$ for almost every $t\in E$.
\end{proof}

\begin{lemma}\label{lem:positive-Kmodulus}
    Let $V \subset X$ be open and connected, and $I,J\subset V$ disjoint non-trivial continua. There are $E\subset\R$, $|E|_1>0$, and a family $\Gamma'=\{\gamma_t:t\in E\}$ satisfying
    
    \begin{enumerate}
        \item every $\gamma_t\in\Gamma'$ is a non-degenerate curve connecting $I$ and $J$ in $V$,
        \item there exists $\varphi\in\Nloc(V,\R)$ such that for every $t\in E$ the curve $\gamma_t\in\Gamma'$ has image in the level set $\varphi^{-1}(t)$, and
        \item $\mod_{K^{-1}}\Gamma'>0$.
    \end{enumerate}
\end{lemma}
\begin{proof}
    Replacing $V$ with a compactly connected subdomain if necessary, we may assume that \begin{equation} \label{uga}
    \int_V  K_f(x) \, d\mathcal{H}^2(x) =K < \infty. 
    \end{equation}
    Fix points $a\in I$ and $b\in J$ and a continuous curve $\eta$ joining $a$ and $b$ in $V$. Define $\varphi\colon X\to\R$ by $\varphi(x)=\dist(x,|\eta|)$. As described in the proof of \cite{Raj:17}*{Proposition 3.5}, we find $\varepsilon'>0$, a set $E_0\subset(0,\varepsilon')$ with $\mathcal{H}^1(E_0)=0$, and for every $t\in E=(0,\varepsilon')\setminus E_0$ a rectifiable injective curve $\gamma_t$ joining $I$ and $J$ in $V$, with image in the level set $\varphi^{-1}(t)$. We set $\Gamma'=\{\gamma_t:t\in E\}$. 
    
    Let $g\colon V\to[0,\infty]$ be admissible for $\Gamma'$. We apply the coarea inequality for Lipschitz maps (Theorem \ref{thm:coarea-Lipschitz}) and Hölder's inequality to obtain
    \begin{align*}
        \varepsilon'&\leq\int_0^{\varepsilon'}\int_{\gamma_t}g\,ds\,dt\leq\frac{4}{\pi}\int_{V}g(x)K_f(x)^{-1/2}K_f(x)^{1/2}\,d\hm^2(x)\\
        &\leq\frac{4}{\pi}\left(\int_{V}K_f(x)\,d\hm^2(x)\right)^{1/2}\left(\int_{V}\frac{g(x)^2}{K_f(x)}\,d\hm^2(x)\right)^{1/2}. 
    \end{align*}
    Combining with \eqref{uga} gives
    $$\mod_{K^{-1}}\Gamma'\geq\left(\frac{\pi\varepsilon'}{4K}\right)^2>0,$$
     where we used that the estimate above holds for all admissible functions.
\end{proof}

If $Z$ is a metric surface, $G \subset Z$ a domain, and $E,F \subset \overline{G}$ disjoint sets, we denote by $\Gamma(E,F;G)$ the family of curves joining $E$ and $F$ in $\overline{G}$. 

\begin{lemma}\label{lem:g-epsilon}
    For any $\varepsilon>0$ the function $g_{\varepsilon}\colon\R^2\to[0,\infty)$ defined by
    $$g_{\varepsilon}(y)=\varepsilon\left(|y|\log\frac{1}{|y|}\log\log\frac{1}{|y|}\right)^{-1} \chi_{\mathbb{D}(0,e^{-2})}$$ 
    is admissible for $\Gamma(\{0\},\partial\D(0,e^{-2});\R^2)$ and 
    $$\int_{\R^2}g_{\varepsilon}(y)^2\log\frac{1}{|y|}\,dy\to0$$
    as $\varepsilon\to0$.
\end{lemma}
\begin{proof}
   Fix $\gamma\in\Gamma(\{0\},\partial\D(0,e^{-2});\R^2)$. We may assume that $\gamma\colon[0,\ell(\gamma)]\to\R^2$ is parametrized by arclength and $\gamma(0)=0$. Then $\ell(\gamma)\geq e^{-2}$ and $|\gamma(t)|\leq t$ for every $0\leq t\leq\ell(\gamma)$. We compute
    \begin{align*}
        \int_{\gamma}g_1\,ds&=\int_0^{\ell(\gamma)}g_1(\gamma(t))\,dt\\
        &=\int_0^{\ell(\gamma)}\left(|\gamma(t)|\log\frac{1}{|\gamma(t)|}\log\log\frac{1}{|\gamma(t)|}\right)^{-1}\,dt\\
        &\geq\int_0^{e^{-2}}\left(t\log\frac{1}{t}\log\log\frac{1}{t}\right)^{-1}\,dt=\infty,
    \end{align*}
    where the last equality follows since $$\frac{d}{ds}\log\log\log \frac{1}{s}=-\left(s\log\frac{1}{s}\log\log\frac{1}{s}\right)^{-1}.$$
    Thus, $g_{\varepsilon}=\varepsilon\cdot g_1$ is admissible for $\Gamma(\{0\},\partial\D(0,e^{-2});\R^2)$ for any $\varepsilon>0$. 
    
    In order to prove the second claim we use polar coordinates and compute
    \begin{align*}
        \int_{\R^2} g_{\varepsilon}(y)^2\log\frac{1}{|y|}\,dy&=\varepsilon^2\int_{\R^2}\left(|y|^2\log\frac{1}{|y|}\left(\log\log\frac{1}{|y|}\right)^2\right)^{-1}\chi_{\D(0,e^{-2})}\,dy\\
        &=\varepsilon^2\int_{0}^{2\pi}\int_0^{e^{-2}}\left(r\log\frac{1}{r}\left(\log\log\frac{1}{r}\right)^2\right)^{-1}\,dr\,d\varphi.
    \end{align*}
    The last term converges to $0$ as $\varepsilon\to0$ since $$\frac{d}{ds}\left(\log\log \frac{1}{s}\right)^{-1}=\left(s\log\frac{1}{s}\left(\log\log\frac{1}{s}\right)^2\right)^{-1}.$$
    The second claim follows. 
\end{proof}

We are now able to prove Proposition \ref{prop:discr-open-nbhd}. Let $V_0$ be the $x_0$-component of $B(x_0,s)$. Denote the $x_0$-component of $f^{-1}(f(x_0)) \cap V_0$ by $J$. We may assume that $V_0 \setminus f^{-1}(f(x_0)) \neq \emptyset$, since otherwise there is nothing to prove. Towards contradiction, assume that $J$ is a non-trivial continuum. Fix another non-trivial continuum $I \subset V_0 \setminus f^{-1}(f(x_0))$. 

By scaling and translating the target we may assume that $f(x_0)=0$, $f(I) \cap \mathbb{D}(0,e^{-2})=\emptyset$, and that the constant $r_0$ in Condition \eqref{ineq:multiplicity} satisfies $r_0\geq e^{-2}$.  Let $\Gamma'$ be the curve family from Lemma \ref{lem:positive-Kmodulus}. Note that $\Gamma=f(\Gamma')$ is a subfamily of $\Gamma(\{0\},\partial\D(0,e^{-2});\R^2)$. Hence, we know from Lemma \ref{lem:g-epsilon} that for any $\varepsilon>0$ the function $ g_{\varepsilon}$ is admissible for $\Gamma$. Lemma \ref{lem:level-set-curves} implies that Lemma \ref{lem:upper-bound-Kmodulus} can be applied to our setting and thus
$$\mod_{K^{-1}}\Gamma'\leq 4 \sqrt{2}\int_{\R^2} g_{\varepsilon}(y)^2N(y,f,B(x_0,s))\,dy.$$

Since $g_\varepsilon$ is symmetric with respect to the origin, combining Assumption \eqref{ineq:multiplicity} with polar coordinates yields 
\begin{align*} 
\int_{\R^2} g_{\varepsilon}(y)^2N(y,f,B(x_0,s))\,dy= 
\int_0^{e^{-2}} r g_{\varepsilon}(r)^2 \int_0^{2\pi} N(re^{i\theta},f,B(x_0,s)) \, d\theta \, dr \\
\leq C \int_0^{e^{-2}} r g_{\varepsilon}(r)^2 \log \frac{1}{r} \, dr =C \int_{\R^2} g_{\varepsilon}(y)^2\log\frac{1}{|y|}\,dy. 
\end{align*}

By the second part of Lemma \ref{lem:g-epsilon}, the right hand integral converges to $0$ as $\varepsilon$ goes to $0$. Thus, $\mod_{K^{-1}}\Gamma'=0$, contradicting Lemma \ref{lem:positive-Kmodulus}. The proof is complete. 

\subsection{Proof of Proposition \ref{prop:ineq-multiplicity}}
Let $x_0$ and $s$ be as in the statement. We may assume that $f(x_0)=0$. We first show that $f^{-1}(y)$ is totally disconnected for most points $y \in f(X)$ around $0$.  

\begin{lemma}\label{lem:2.1} 
Let $\beta'$ be the set of those $0\leq \theta < 2\pi$ for which there is $R_\theta>0$ so that $f^{-1}(R_\theta e^{i\theta})$ contains a non-degenerate continuum. Then $|\beta'|_1=0$. 
\end{lemma}
\begin{proof}
    We define 
    \[
    \varphi\colon X \setminus f^{-1}(0)\to \mathbb{S}^1, \quad \varphi(x)=\frac{f(x)}{|f(x)|}, 
    \]
    and note that $\rho_f^u/|f|$ is a weak upper gradient of $\varphi$. Towards a contradiction we assume that $|\beta'|_1>0$. Then there are $\delta,\varepsilon>0$ and a set $\beta_{\delta}'\subset\beta'$, $|\beta_{\delta}'|_1>0$, such that for every $\theta\in\beta_{\delta}'$ there exists $R_{\theta}\in[\varepsilon,1]$ for which 
    $f^{-1}(R_{\theta}e^{i\theta})$ contains a continuum $E_\theta$ with 
    $\mathcal{H}^1(E_\theta)\geq\delta$. As in the proof of Lemma \ref{lem:level-set-curves}, we see that almost every 
    $\theta \in\beta_{\delta}'$ the continuum $E_\theta$ is the image of a rectifiable curve $\gamma_\theta$, and the modulus of the family of such curves is positive. By the definition of lower gradients and since $f\circ \gamma_\theta$ is constant by construction, we then have that $\rho^l_f=0$ almost everywhere in  $$E=\bigcup_{\theta\in\beta_\delta'} E_\theta. $$ 
    Furthermore, since $f$ has finite distortion, also $\rho^u_f=0$ almost everywhere in $E$. Let 
    \[
    F=\{x \in X: |f(x)| \geq \varepsilon, \,\rho^u_f(x)=0\} \supset E. 
    \]
    We apply the Sobolev coarea inequality (Theorem \ref{thm:coarea-Sobolev}) to compute
    \begin{align*}
        0<\delta |\beta_{\delta}'|_1\leq\int\displaylimits^*_{\beta_\delta'}\mathcal{H}^1(E_\theta)\,d\theta\leq \frac{4}{\pi}\int_F\frac{\rho_f^u}{|f|}\,d\hm^2=0,  
    \end{align*}
    a contradiction. The proof is complete. 
\end{proof}

\begin{lemma}\label{lem:2.2}
    Let $\beta'$ be the set in Lemma \ref{lem:2.1}. There exists $\beta\supset\beta'$ with 
    $|\beta|_1=0$, and an open $\Omega'\subset X$, such that 
    \begin{enumerate}
        \item $f|_{\Omega'}$ is a local homeomorphism, and
        \item if $V=\{te^{i\theta}:t>0,\theta\in\beta\}$, then $\Omega'\supset X\setminus f^{-1}(V)$.
    \end{enumerate}
\end{lemma}

\begin{proof}
    Set $V'=\{te^{i\theta}:\theta\in\beta',t>0\}$. Let $y\in f(X)\setminus V'$ and $x\in f^{-1}(y)$. Then, since $\{x\}$ is a component of $f^{-1}(y)$, there is a Jordan domain $\widetilde{U}_x$ in $X$ such that $x\in \widetilde{U}_x$ and $y\notin f(\partial\widetilde{U}_x)$. Let $W_x$ be the $y$-component of $\R^2\setminus f(\partial\widetilde{U}_x)$ and $U_x$ the $x$-component of $f^{-1}(W_x)$. It follows that $f(\partial U_x)\subset\partial W_x$. Indeed, otherwise there is a point $a\in\partial U_x$ with $f(a)\in W_x$ and therefore there exists a neighbourhood $Y$ of $f(a)$ in $W_x$, but the $a$-component of $f^{-1}(Y)$ is not contained in $U_x$, which is a contradiction. 
    
    The assumption that $f$ is sense-preserving now implies $f(\partial U_x)=\partial W_x$. Using basic degree theory, we conclude that $f^{-1}(z)$ has at most $\deg(y,f,U_x)$ 
    components in $U_x$ for every $z \in W_x$. Furthermore, arguing as in the proof of Lemma \ref{lem:2.1} we see that for almost every such $z$ all of these components are points. In other words, 
    $$
    N(z,f,U_x)\leq \deg(y,f,U_x) < \infty 
    $$
    for almost every $z \in W_x$. In particular, every $x \in U_x$ satisfies the conditions in Proposition \ref{prop:discr-open-nbhd}, and therefore $f|_{U_x}$ is open and discrete. 
    
    We have established the following.
    \begin{enumerate}[label=(\roman{enumi})]
        \item If $y\in f(X)\setminus V'$ and $x\in f^{-1}(y)$, then $x$ has a neighbourhood $U_x$ such that $f|_{U_x}$ is open and discrete. \label{claim:1}
    \end{enumerate}
    We define $$\widehat{\Omega}=\{x\in X:x\text{-component of }f^{-1}(f(x))\text{ is } \{x\}\}.$$ Note that if $x \in \widehat{\Omega}$, then there exists a neighbourhood $Y$ of $f(x)$ such that the closure of the $x$-component of $f^{-1}(Y)$ is compact. As above, we find a neighbourhood $U_x$ of $x$ such that $f|_{U_x}$ is open and discrete. In particular, $\widehat{\Omega}$ is open. Moreover, it follows from \ref{claim:1} that $\widehat{\Omega}\supset X \setminus f^{-1}(V')$. We have shown that
    \begin{enumerate}[label=(\roman{enumi})]\setcounter{enumi}{1}
        \item $\widehat{\Omega}$ is open, $f|_{\widehat{\Omega}}$ is open and discrete, and $\widehat{\Omega}\supset X \setminus f^{-1}(V')$. \label{claim:2}
    \end{enumerate}
    Denote by $\mathcal{B}_f$ the branch set of $f|_{\widehat{\Omega}}$, i.e., the set of points where 
    $f|_{\widehat{\Omega}}$ fails to be locally invertible, and define $$\beta''=\{0\leq\theta<2\pi:Re^{i\theta}\in f(B_f)\text{ for some }R>0\}.$$ Recall that $\mathcal{B}_f$ is closed and countable, thus $\beta''$ is countable. It follows from Lemma \ref{lem:2.1} and \ref{claim:2} that the sets $\Omega'=\widehat{\Omega}\setminus\mathcal{B}_f$ and $\beta=\beta'\cup\beta''$ possess the desired properties.
\end{proof}

\begin{lemma}\label{lem:2.3}
    Let $m\in\N$, $0<r<e^{-2}$, and assume that $\overline{B}(x_0,2s)$ is compact and satisfies  $f(\overline{B}(x_0,2s)) \subset \mathbb{D}(0,1)$. If 
    $$
    E_m=\{0\leq\theta<2\pi:N(re^{i\theta},f,B(x_0,s))=m\},
    $$ then
    $$
    m|E_m|_1\leq \frac{64 \sqrt{2}}{\pi s^2}\int_{F_m}K_f\,d\hm^2\cdot\log\frac{1}{r},
    $$
    where $F_m=\{x\in X:\arg(f(x))\in E_m\}$.
\end{lemma}
\begin{proof}
    We assume $|E_m|_1>0$, otherwise there is nothing to show. Let $\beta$ and $\Omega'$ be as in Lemma \ref{lem:2.2}. We set $E_m'=E_m\setminus\beta$ and note that $|E_m'|_1=|E_m|_1$ since $|\beta|_1=0$. We also denote 
    $$
    F'_m=\{x\in X:\arg(f(x))\in E'_m\} \subset F_m. 
    $$
    Fix $\theta\in E_m'$, then
    $$f^{-1}(\{te^{i\theta}:t\geq r\})\subset\Omega'.$$ 
    We can therefore apply path lifting of local homeomorphisms to curves $I_\theta=\{te^{i\theta}:r\leq t \leq 1\}$ as follows: if $\{x_1,...,x_m\}=f^{-1}(re^{i\theta})\cap B(x,s)$ then for every $j\in\{1,...,m\}$ there exists a maximal lift $\gamma_\theta^j$ of $I_\theta$ starting at $x_j$, see \cite{Ric93}*{Theorem II.3.2}. Note that if $\varphi\colon X\to[0,2\pi)$ is defined by $\varphi(x)=\arg(f(x))$, then the image of each $\gamma_\theta^j$ is contained in the level set $\varphi^{-1}(\theta)$. 
    
    Since $\overline{B}(x,2s)$ is compact and $f(\overline{B}(x,2s)) \subset \mathbb{D}(0,1)$, every curve $\gamma_\theta^j$ connects $B(x,s)$ and $X\setminus B(x,2s)$, and so  $\mathcal{H}^1(|\gamma_\theta^j|)\geq s$. Moreover, $f|_{|\gamma_\theta^j|}$ is injective. It follows that  
    \begin{equation} \label{kasha}
    s\cdot m \leq \sum_{j=1}^m \mathcal{H}^1(|\gamma_\theta^j|) \leq \mathcal{H}^1(\{x\in X:\arg(f(x))=\theta\})
    \end{equation}  
    for every $\theta \in E'_m$. 
    
    We combine \eqref{kasha} with the Sobolev coarea inequality (Theorem \ref{thm:coarea-Sobolev}) and Hölder's inequality to compute
    \begin{align*}
        s\cdot m \cdot |E_m|_1&=s \cdot m \cdot |E_m'|_1\\
        &\leq\int_{E_m'}\mathcal{H}^1(\{x\in X:\arg(f(x))=\theta\})\,d\theta\\
        &\leq\frac{4}{\pi}\int_{F_m}\frac{\rho_f^u}{|f|}\,d\hm^2\leq\frac{4}{\pi}\int_{F_m}K_f^{1/2}\cdot \frac{(\rho_f^u \cdot \rho_f^l)^{1/2}}{|f|}\,d\hm^2\\
        &\leq\frac{4}{\pi}\left(\int_{F_m}K_f\,d\hm^2\right)^{1/2}\bigg(\underbrace{\int_{F'_m}\frac{\rho_f^u \cdot \rho_f^l}{|f|^2}\,d\hm^2}_{=:I}\bigg)^{1/2}.
    \end{align*}
    For each $j\in\{1,...,m\}$ we define the curve family
    $$\Gamma'_j=\{\gamma_\theta^j:t\in E_m'\}. $$
    Lemma \ref{lem:level-set-curves} applied to $\Gamma_j'$ shows that $\mathcal{H}^1(|\gamma_\theta^j|\cap X_0)=0$ for almost every $\theta\in E_m'$ and every $j\in\{1,...,m\}$, where $X_0$ is as in Theorem \ref{thm:area-ineq}. Hence, if 
    $$
    F''_m=\{x \in X: \, x\in|\gamma_\theta^j|\text{ for some }\theta\in E_m'\text{ and }1\leq j\leq m\} \supset F'_m, 
    $$
    then $\hm^2(F''_m \cap X_0)=0$ and $N(y,f,F''_m)\leq m$ for every $y\in\R^2$. By the area inequality (Theorem \ref{thm:area-ineq}) and polar coordinates, 
    \begin{align*}
        I\leq 4 \sqrt{2}\int_{E_m}\int_r^1 \frac{N(se^{i\theta},f,F''_m)}{s}\,ds \, d\theta \leq 4\sqrt{2}\cdot |E_m|_1 \cdot m \cdot \log\frac{1}{r}. 
    \end{align*}
    The lemma follows by combining the estimates. 
\end{proof}

    Proposition \ref{prop:ineq-multiplicity} follows from Lemma \ref{lem:2.3}: notice that by scaling we may assume that $f(\overline{B}(x_0,2s)) \subset \mathbb{D}(0,1)$, so that the conditions of Lemma \ref{lem:2.3} are satisfied. Recall that the sets $F_m$ are pairwise disjoint. Therefore, summing the estimate in Lemma \ref{lem:2.3} over $m$ gives 
    \begin{align*}
        \int_0^{2\pi}N(re^{i\theta},f,B(x_0,s))\,d\theta&=\sum_{m=1}^{\infty}m|E_m|_1\\
        &\leq C\log\frac{1}{r}\sum_{m=1}^{\infty}\int_{F_m}K_f(x)\,d\hm^2\\
        &\leq C\log\frac{1}{r}\int_{X}K_f(x)\,d\hm^2. 
    \end{align*}
    We may replace $X$ with a compactly contained subdomain if necessary to guarantee that $K_f$ is integrable. 
    Proposition \ref{prop:ineq-multiplicity} follows. 


\section{Regularity of the inverse} \label{sec:regul}
In this section we study the regularity of the inverse of a mapping of finite distortion and prove Theorem \ref{thm:inverse-Sobolev}. Let $f\in\Nloc(X,\Omega')$ be a homeomorphism with $K_f\in\Lloc^1(X)$, where $\Omega' \subset \R^2$. We set $\phi=f^{-1}\colon\Omega'\to X$ and define $\psi\colon\Omega'\to[0,\infty]$ by $$\psi(y)=\frac{1}{\rho_f^l(\phi(y))}.$$ 

\begin{lemma}\label{lem:psi-L2}
We have 
$$
\int_E \psi(y)^2 \, dy \leq 2 \int_{\phi(E)} K_f(x) \, d\hm^2(x) 
$$
for every Borel set $E \subset \Omega'$. In particular, $\psi \in \Lloc^2(\Omega')$. 
\end{lemma}

\begin{proof}
Again, let $u:U \to X$, $U \subset \R^2$, be a weakly $(4/\pi)$-quasiconformal parametrization and $h=f\circ u$. Then $h$ is locally in $N^{1,2}(U,\R^2)$ and monotone. Therefore, $h$ satisfies Condition $(N)$ and consequently the euclidean area formula, see \cite{MalMar95}. Combining the area formula with distortion estimates established in previous sections, we have 

\begin{eqnarray*} 
\int_{E} \psi(y)^2 \, dy &=& 
\int_{h^{-1}(E)} \frac{J(\apmd h_z)}{\rho_f^l(u(z))^2}\, dz 
= \int_{h^{-1}(E)} \frac{L_h(z)\cdot l_h(z)}{\rho_f^l(u(z))^2}\, dz \\
&\leq& \int_{h^{-1}(E)} \frac{\rho_f^u(u(z))\cdot \rho_f^l(u(z))}{\rho_f^l(u(z))^2} L_u(z)^2\, dz \\
&\leq& 2 \int_{h^{-1}(E)} K_f(u(z)) \cdot J(\apmd u_z) \, dz. \end{eqnarray*} 
Here the second equality holds since both the domain and target of $h$ are euclidean domains and the first inequality holds by Lemma \ref{lem:compo} and Proposition \ref{prop:mildest}. The second inequality holds by \eqref{eq:John} and recalling that we can choose $u$ so that the John ellipses of $\apmd u_z$ are disks for almost every $z$. The claim now follows from the area formula for $u$ (Theorem \ref{thm:area-formula}). 
\end{proof}

\begin{lemma}\label{lem:psi-upper-grad}
Suppose $\alpha:X \to \mathbb{R}$ is $1$-Lipschitz. Then $v=\alpha \circ \phi$ 
is absolutely continuous on almost every line parallel to coordinate axes, and $|\partial_j v| \leq \frac{16\sqrt{2}}{\pi} \cdot \psi$ almost everywhere for $j=1,2$. 
\end{lemma}

\begin{proof}
    It suffices to consider horizontal lines. Fix a square $Q \subset \Omega'$ with sides parallel to coordinate axes. By scaling and translating, we may assume that $Q=[0,1]^2$. 
    
     By Lebesgue's theorem, there exists a set $\Phi\subset (0,1)$ of full measure so that if $s_0\in\Phi$ then
    \begin{equation} \label{eq:sshu} 
    \frac{1}{2\varepsilon}\int_{F_\varepsilon}\psi(y)\,dy=\frac{1}{2\varepsilon}\int_{s_0-\varepsilon}^{s_0+\varepsilon}\int_{t_1}^{t_2}\psi(t,s)\,dt\,ds\to\int_{t_1}^{t_2}\psi(t,s_0)\,dt 
    \end{equation}
    as $\varepsilon\to0$ for every $0\leq t_1<t_2\leq1$, where $F_\varepsilon=[t_1,t_2]\times[s_0-\varepsilon,s_0+\varepsilon]$.

    Fix $s_0 \in \Phi$. The claim now follows from Lemma \ref{lem:psi-L2} if we can show that 
    \begin{equation} \label{eq:phoe}
    |\phi(t_2,s_0)-\phi(t_1,s_0)| \leq \frac{16 \sqrt{2}}{\pi} \int_{t_1}^{t_2} \psi(t,s_0) \, dt  
    \end{equation}
    for every $0 \leq t_1<t_2 \leq 1$. 
    
    Given $0<\varepsilon<\min\{s_0,1-s_0\}$ we set $E_\varepsilon=\phi(F_\varepsilon)$.
    Let $\varphi=\pi_2\circ f|_{E_{\varepsilon}}$, where $\pi_2$ denotes projection to the $s$-axis on the $(t,s)$-plane. By continuity of $\varphi$, Lemma \ref{lem:level-set-curves}, and the Sobolev coarea inequality (Theorem \ref{thm:coarea-Sobolev}) applied to $\varphi$, we have 
    \begin{align*}
        |\phi(t_2,s_0)-\phi(t_1,s_0)|&\leq \delta(\varepsilon)+\frac{1}{2\varepsilon}\int_{s_0-\varepsilon}^{s_0+\varepsilon}\mathcal{H}^1(\varphi^{-1}(s)\setminus X_0)\,ds\\
        &\leq\delta(\varepsilon)+\frac{2}{\pi\varepsilon}\int_{E_\varepsilon \setminus X_0}\frac{\rho_f^u\cdot\rho_f^l}{\rho_f^l}\chi_{\rho_f^l\neq0}\,d\hm^2,  
    \end{align*}
    where $X_0$ is the set in the Area inequality (Theorem \ref{thm:area-ineq}) and $\delta(\varepsilon)\to0$ as $\varepsilon\to0$. 
    Combining with Theorem \ref{thm:area-ineq},  we obtain
    \begin{equation} \label{eq:tua}
    |\phi(t_2,s_0)-\phi(t_1,s_0)|\leq 
    \delta(\varepsilon)+\frac{8\sqrt{2}}{\pi \varepsilon} 
    \int_{F_\varepsilon} \psi(y) \, dy. 
    \end{equation}   
    Now \eqref{eq:phoe} follows by combining \eqref{eq:tua} and \eqref{eq:sshu}. 
\end{proof}

We are ready to prove Theorem \ref{thm:inverse-Sobolev}. Since $\phi$ is continuous, $d_X(\phi(\cdot),x_0) \in \Lloc^2(\Omega')$ for every $x_0 \in X$. By Lemma \ref{lem:psi-L2} and the ACL-characterization of Sobolev functions (see \cite{HKST:15}*{Theorem 6.1.17}), we see that every $v$ in Lemma \ref{lem:psi-upper-grad} belongs to $W^{1,2}_{\text{loc}}(\Omega')$ and satisfies $|\nabla v|\leq \frac{32\psi}{\pi}$ 
almost everywhere. Furthermore, the characterization of Sobolev maps in terms of post-compositions with $1$-Lipschitz functions, i.e., in terms of the functions $v$ above (see \cite{HKST:15}*{Theorem 7.1.20 and Proposition 7.1.36}), shows that $\phi \in N^{1,2}_{\text{loc}}(\Omega',X)$. The proof is complete. \\

\begin{rmk} When $X \subset \mathbb{R}^2$, the $N_{\operatorname{loc}}^{1,2}(X,\R^2)$-regularity assumption in Theorem \ref{thm:inverse-Sobolev} may be replaced with  $f \in N_{\operatorname{loc}}^{1,1}(X,\R^2)$. Moreover, the conclusion on the regularity of $f^{-1}$ is more precise, see \cite{HenKos06}. While our results only concern $N_{\operatorname{loc}}^{1,2}$-maps, it would be interesting to extend the definition of finite distortion to $N_{\operatorname{loc}}^{1,1}$-maps between metric surfaces and develop basic properties including improvements of Theorem \ref{thm:inverse-Sobolev}. One cannot expect the conclusions of Remarks \ref{rem:sense} and \ref{rmk:NN} to hold in the $N^{1,1}$-setting without additional assumptions; maps $f \in N^{1,1}_{\operatorname{loc}}(X,\R^2)$ of finite distortion need not be continuous or satisfy Condition (N) even when $X\subset \R^2$ (see e.g. \cite{HenKos14}). 
\end{rmk}


\section{Reciprocal surfaces}\label{section:reciprocal}

Recall the \emph{geometric definition of quasiconformality}: a homeomorphism $f\colon X\to Y$ is \emph{quasiconformal} if there exists $C\geq 1$ such that 
\begin{equation} \label{eq:geomQC}
C^{-1}\mod f(\Gamma)\leq \mod \Gamma \leq C \mod f(\Gamma) \end{equation} 
for each curve family $\Gamma$ in $X$. 

We say that metric surface $X$ is \textit{reciprocal} if there exists $\kappa>0$ such that for every topological quadrilateral $Q\subset X$ and for the families $\Gamma(Q)$ and $\Gamma^*(Q)$ of curves joining opposite sides of $Q$ we have 
$$
\mod  \Gamma(Q)\cdot \mod\Gamma^*(Q)\leq\kappa. 
$$ 
If $X$ is reciprocal, $x\in X$ and $R>0$ so that  $X\setminus B(x,R)\not=\emptyset$, then by \cite{NR22}*{Theorem 1.8} we have 
\begin{align}\label{eq:reciprocal-1}
 \lim_{r\to 0} \mod\Gamma(B(x,r), X\setminus B(x,R); X) = 0. 
\end{align}
Recall that $\Gamma(E,F;G)$ is the family of curves joining $E$ and $F$ in $\overline{G}$. 

Reciprocal surfaces are the metric surfaces that admit quasiconformal parametrizations by euclidean domains, see \cite{Raj:17}, \cite{Iko:19}, \cite{NR22}. See \cite{Raj:17}, \cite{RR19}, \cite{EBPC21}, \cite{MW21}, \cite{NR:21} and \cite{NR22} for further properties of reciprocal surfaces. 

It is desirable to find non-trivial conditions which imply reciprocality. For instance, one could hope that the existence of maps satisfying the conditions of Theorem \ref{thm:main} forces $X$ to be reciprocal. However, this is not the case. 

\begin{prop}\label{prop:X-not-reciprocal}
    Given an increasing $\phi:[1,\infty) \to [1,\infty)$ so that $\phi(t)\to\infty$ as $t\to\infty$, there is a non-reciprocal metric surface $X$ and a homeomorphism $f\colon X\to\R^2$ so that $f\in \Nloc(X,\R^2)$ and $\phi(K_f)$ is locally integrable.  
\end{prop}

The map $f_0$ defined in the proof below is known as Ball's map (\cite{Bal81}) and illustrates that the integrability condition in Theorem \ref{thm:main} is sharp.

\begin{proof}
Let $f_0\colon \R^2\to\R^2$ be defined by $f_0(x,y)=(x,\eta(x,y))$, where 
\begin{eqnarray*}
\eta(x,y)=
\left\{ 
\begin{array}{ll}
|x|y, &  0\leq |x| \leq 1, \, 0 \leq |y| \leq 1, \\ 
(2(|y|-1)+|x|(2-|y|))\frac{y}{|y|}, & 0 \leq |x| \leq 1, \, 1 \leq |y| \leq 2, \\ 
y, & \text{otherwise}. 
\end{array}
\right. 
\end{eqnarray*}

Note that $f_0$ is not open and discrete since it maps the segment $I=\{0\}\times [-1,1]$ to the origin. Also, $f_0$ is the identity outside $(-1,1)\times (-2,2)$. 

Calculating the Jacobian matrix shows that $f_0$ is sense-preserving and Lipschitz, 
$K_{f_0}$ is bounded outside $(-1,1)\times (-1,1)$, 
and 
\begin{equation} \label{darke}
K_{f_0}(x,y)=\frac{1}{|x|} \quad \text{for all } (x,y) \in (-1,1)\times (-1,1). 
\end{equation} 
It follows that $K_{f_0}$ is not in $L^1_{\text{loc}}(\R^2)$ but $K_{f_0}\in L^p_{\text{loc}}(\R^2)$ for every $0<p<1$. 

    We change the metric on $\R^2$ to obtain the desired metric surface $X$ and $f\colon X\to\R^2$. Define $\omega:\R^2 \to [0,1]$ by $\omega(z)=1$ when $\dist(z,I)\geq 1$ and by 
\begin{equation} 
\label{eq:ikki}
    \omega(z)=\frac{1}{\phi(\dist(z,I)^{-1})} 
\end{equation} 
    otherwise, where $I=\{0\}\times [-1,1]$. Moreover, let 
    $$
    d_\omega(x,y):=\inf_{\gamma} \int_{\gamma}\omega\, ds, 
    $$
    where the infimum is taken over all rectifiable curves $\gamma$ connecting $x,y\in \R^2$. 
    
    Now $X=(\R^2/I,d_\omega)$ is homeomorphic to $\R^2$ and has locally finite $\mathcal{H}^2$-measure. Let $\pi\colon \R^2\to \R^2/I$ be the projection map, $\id_\omega\colon \R^2/I\to X$ the identity, and $\pi_\omega\colon \R^2\to X$, $\pi_\omega=\id_\omega\circ\pi$. 
    
    Since modulus is conformally invariant and $\omega$ is a conformal change of metric outside $I$, the family of curves joining any non-trivial continuum and the point 
    $p:=\pi_\omega(I)$ in $X$ has positive modulus. By \eqref{eq:reciprocal-1}, it follows that 
    $X$ is non-reciprocal.

 We define $f:X \to \R^2$ by $f:=f_0\circ\pi_\omega^{-1}$. Then $f$ is absolutely continuous on almost every rectifiable curve in $X$, and $\rho_f^u(z) \leq (\omega(z))^{-1} \cdot L$ 
    for almost every $z \in X$, where $L$ is the Lipschitz constant of $f_0$. Therefore, 
    $$
    \int_E (\rho_f^u)^2 \, d \hm^2 \leq L^2 |\pi_\omega^{-1}(E)|_2
    $$
    for every Borel set $E \subset X$. We conclude that 
    $f \in \Nloc(X,\R^2)$. 
    
    It remains to estimate the integral of $\phi(K_f)$. To this end, notice that since $\omega$ is a conformal change of metric, we have 
    $$K_f(z)=K_{f_0}(\pi_\omega^{-1}(z))$$ for almost every $z \in X$. Therefore, it suffices to check that $\phi(K_f)$ is integrable over $E=\pi_\omega((-1,1)\times(-1,1))$. By \eqref{darke} and \eqref{eq:ikki}, we have
    $$
    \int_E \phi(K_f(z)) \,d\hm^2=
    \int_{(-1,1)^2} \phi(K_{f_0}) \cdot \omega^2 \, dx \, dy \leq  
    \int_{(-1,1)^2} \frac{1}{\phi(|x|^{-1})} \, dx \,dy <\infty.  $$
   The proof is complete. 
\end{proof}

We prove in \cite{MR24}*{Theorem 1.3} that if there is a non-constant $f \in \operatorname{FDP}(X,\R^2)$ (not necessarily a homeomorphism) with \emph{bounded} distortion, then $X$ is reciprocal. We also show (see \cite{MR24}*{Corollary 1.2}) that the geometric definition \eqref{eq:geomQC} is quantitatively equivalent with the path definition (requiring $K_f$ to be bounded) of quasiconformality for homeomorphisms $f:X \to \R^2$. By Williams' theorem \cite{Wil:12}, the equivalence between the analytic (requiring $C(x)$ to be bounded in \eqref{ineq:analytic-distortion}) and geometric definitions of quasiconformality for homeomorphisms holds in even greater generality.


\section{Existence of maximal weak lower gradients}\label{section:lower-gradient}

Let $X$ and $Y$ be metric surfaces. We now complete the discussion in Section \ref{section:Sobolev} by proving that each $f\in \Nloc(X,Y)$ has a \emph{maximal weak lower gradient}. Precisely, we claim that there is a weak lower gradient $\rho_f^l$ of $f$ so that if $\rho^l$ is another weak lower gradient of $f$ 
then 
$$
\rho_f^l(x)\geq\rho^l(x) \quad \text{for almost every } x \in X. 
$$ 
Moreover, $\rho^l_f$ is unique up to a set of measure zero. The proof of these facts is analogous to the existence of minimal weak upper gradients, see \cite{HKST:15}*{Theorem 6.3.20}.

First, recall that $f$ is absolutely continuous along almost every curve \cite{HKST:15}*{Lemma 6.3.1}. It follows from \cite{HKST:15}*{Lemma 5.2.16} that if $\rho$ is a weak lower gradient of $f$ and $\sigma\colon X\to[0,\infty]$ is a Borel function such that $\sigma=\rho$ almost everywhere in $X$, then $\sigma$ is a weak lower gradient of $f$. In particular, if $E\subset X$ is Borel and satisfies $\hm^2(E)=0$ then $\rho\chi_{X\setminus E}$ is a weak lower gradient of $u$, compare to \cite{HKST:15}*{Lemma 6.2.8}. We conclude that if there exists a maximal weak lower gradient $\rho_f^l$ of $f$, it has to be unique up to sets of measure zero. 

To prove existence of $\rho^l_f$, we may assume without loss of generality that $\mathcal{H}^2(X)<\infty$. Arguing exactly as in the proof of \cite{HKST:15}*{Lemma 6.3.8}, we can show that if $\sigma,\tau\in L^2(X)$ are weak lower gradients of a map $f\colon X\to Y$ that is absolutely continuous along almost every curve in $X$ and if $E$ is a measurable subset of $X$ then the function 
    $$\rho=\sigma\cdot\chi_E+\tau\cdot\chi_{X\setminus E}$$
is a weak lower gradient of $f$. Now, by choosing $E=\{x\in X:\sigma>\tau\}$, it follows that $\rho\colon X\to[0,\infty]$ defined by
        $$\rho(x)=\max\{\sigma(x),\tau(x)\}$$
is a 2-integrable weak lower gradient of $f$. After applying Fuglede's lemma, see e.g.\ \cite{HKST:15}*{Section 5.1}, we established the following lemma. 

\begin{lemma}\label{lem:max-is-lower-gradient}
    If $f\colon X\to Y$ is absolutely continuous along almost every curve, then the collection $\mathcal{L}$ of $2$-integrable weak lower gradients of $f$ is closed under pointwise maximum operations.
\end{lemma}

Let $(\rho_i)\subset\mathcal{L}$ be a sequence such that
    $$\lim_{i\to\infty}||\rho_i||_{L^2}=\sup\{||\rho||_{L^2}:\rho\in\mathcal{L}\}.$$
    By Lemma \ref{lem:max-is-lower-gradient}, the sequence $(\rho_i')$ given by $\rho_i'(x)=\max_{1\leq j\leq i}\rho_j(x)$ is in $\mathcal{L}$. Note that $(\rho_i')$ is pointwise increasing. The limit function
    $$\rho_f^l:=\lim_{i\to\infty}\rho_i'$$
    is Borel by \cite{HKST:15}*{Proposition 3.3.22}. The monotone convergence theorem implies that $\rho_i'\to\rho_f^l$ in $L^2(X)$ and by Fuglede's lemma $\rho_f^l\in\mathcal{L}$, see e.g.\ \cite{HKST:15}*{Section 5.1}. By construction, $\rho_f^l$ is a maximal weak lower gradient of $f$. The proof is complete.



\begin{bibdiv}
\begin{biblist}

\bib{AIM09}{book}{
      author={Astala, Kari},
      author={Iwaniec, Tadeusz},
      author={Martin, Gaven},
       title={Elliptic partial differential equations and quasiconformal
  mappings in the plane},
      series={Princeton Mathematical Series},
   publisher={Princeton University Press, Princeton, NJ},
        date={2009},
      volume={48},
        ISBN={978-0-691-13777-3},
      review={\MR{2472875}},
}

\bib{Bal81}{article}{
      author={Ball, J.~M.},
       title={Global invertibility of {S}obolev functions and the
  interpenetration of matter},
        date={1981},
        ISSN={0308-2105,1473-7124},
     journal={Proc. Roy. Soc. Edinburgh Sect. A},
      volume={88},
      number={3-4},
       pages={315\ndash 328},
         url={https://doi.org/10.1017/S030821050002014X},
      review={\MR{616782}},
}

\bib{Bal97}{incollection}{
      author={Ball, Keith},
       title={An elementary introduction to modern convex geometry},
        date={1997},
   booktitle={Flavors of geometry},
      series={Math. Sci. Res. Inst. Publ.},
      volume={31},
   publisher={Cambridge Univ. Press, Cambridge},
       pages={1\ndash 58},
}

\bib{BonKle02}{article}{
      author={Bonk, Mario},
      author={Kleiner, Bruce},
       title={Quasisymmetric parametrizations of two-dimensional metric
  spheres},
        date={2002},
        ISSN={0020-9910},
     journal={Invent. Math.},
      volume={150},
      number={1},
       pages={127\ndash 183},
         url={https://doi.org/10.1007/s00222-002-0233-z},
      review={\MR{1930885}},
}

\bib{EBPC21}{article}{
      author={Eriksson-Bique, Sylvester},
      author={Poggi-Corradini, Pietro},
       title={On the sharp lower bound for duality of modulus},
        date={2022},
        ISSN={0002-9939},
     journal={Proc. Amer. Math. Soc.},
      volume={150},
      number={7},
       pages={2955\ndash 2968},
         url={https://doi.org/10.1090/proc/15951},
      review={\MR{4428881}},
}

\bib{EG92}{book}{
      author={Evans, Lawrence~C.},
      author={Gariepy, Ronald~F.},
       title={Measure theory and fine properties of functions},
      series={Studies in Advanced Mathematics},
   publisher={CRC Press, Boca Raton, FL},
        date={1992},
}

\bib{EH21}{article}{
      author={Esmayli, Behnam},
      author={Haj\l~asz, Piotr},
       title={The coarea inequality},
        date={2021},
        ISSN={2737-0690,2737-114X},
     journal={Ann. Fenn. Math.},
      volume={46},
      number={2},
       pages={965\ndash 991},
         url={https://doi.org/10.5186/aasfm.2021.4654},
      review={\MR{4307012}},
}

\bib{EIR22}{article}{
      author={Esmayli, Behnam},
      author={Ikonen, Toni},
      author={Rajala, Kai},
       title={Coarea inequality for monotone functions on metric surfaces},
        date={2023},
        ISSN={0002-9947,1088-6850},
     journal={Trans. Amer. Math. Soc.},
      volume={376},
      number={10},
       pages={7377\ndash 7406},
         url={https://doi.org/10.1090/tran/8998},
      review={\MR{4636694}},
}

\bib{HenKos06}{article}{
      author={Hencl, Stanislav},
      author={Koskela, Pekka},
       title={Regularity of the inverse of a planar {S}obolev homeomorphism},
        date={2006},
        ISSN={0003-9527},
     journal={Arch. Ration. Mech. Anal.},
      volume={180},
      number={1},
       pages={75\ndash 95},
         url={https://doi.org/10.1007/s00205-005-0394-1},
      review={\MR{2211707}},
}

\bib{HeiKei11}{article}{
      author={Heinonen, Juha},
      author={Keith, Stephen},
       title={Flat forms, bi-{L}ipschitz parameterizations, and smoothability
  of manifolds},
        date={2011},
        ISSN={0073-8301},
     journal={Publ. Math. Inst. Hautes \'{E}tudes Sci.},
      number={113},
       pages={1\ndash 37},
         url={https://doi.org/10.1007/s10240-011-0032-4},
      review={\MR{2805596}},
}

\bib{HenKos14}{book}{
      author={Hencl, Stanislav},
      author={Koskela, Pekka},
       title={Lectures on mappings of finite distortion},
      series={Lecture Notes in Mathematics},
   publisher={Springer, Cham},
        date={2014},
      volume={2096},
        ISBN={978-3-319-03172-9; 978-3-319-03173-6},
         url={https://doi.org/10.1007/978-3-319-03173-6},
      review={\MR{3184742}},
}

\bib{HKST:15}{book}{
      author={Heinonen, Juha},
      author={Koskela, Pekka},
      author={Shanmugalingam, Nageswari},
      author={Tyson, Jeremy~T.},
       title={Sobolev spaces on metric measure spaces: An approach based on
  upper gradients},
      series={New Mathematical Monographs},
   publisher={Cambridge University Press, Cambridge},
        date={2015},
      volume={27},
}

\bib{HeiRic02}{article}{
      author={Heinonen, Juha},
      author={Rickman, Seppo},
       title={Geometric branched covers between generalized manifolds},
        date={2002},
        ISSN={0012-7094},
     journal={Duke Math. J.},
      volume={113},
      number={3},
       pages={465\ndash 529},
         url={https://doi.org/10.1215/S0012-7094-02-11333-7},
      review={\MR{1909607}},
}

\bib{HenRaj13}{article}{
      author={Hencl, Stanislav},
      author={Rajala, Kai},
       title={Optimal assumptions for discreteness},
        date={2013},
        ISSN={0003-9527,1432-0673},
     journal={Arch. Ration. Mech. Anal.},
      volume={207},
      number={3},
       pages={775\ndash 783},
         url={https://doi.org/10.1007/s00205-012-0574-8},
      review={\MR{3017286}},
}

\bib{HeiSul02}{article}{
      author={Heinonen, Juha},
      author={Sullivan, Dennis},
       title={On the locally branched {E}uclidean metric gauge},
        date={2002},
        ISSN={0012-7094},
     journal={Duke Math. J.},
      volume={114},
      number={1},
       pages={15\ndash 41},
         url={https://doi.org/10.1215/S0012-7094-02-11412-4},
      review={\MR{1915034}},
}

\bib{IKO01}{article}{
      author={Iwaniec, Tadeusz},
      author={Koskela, Pekka},
      author={Onninen, Jani},
       title={Mappings of finite distortion: monotonicity and continuity},
        date={2001},
        ISSN={0020-9910},
     journal={Invent. Math.},
      volume={144},
      number={3},
       pages={507\ndash 531},
         url={https://doi.org/10.1007/s002220100130},
      review={\MR{1833892}},
}

\bib{Iko:19}{article}{
      author={Ikonen, Toni},
       title={Uniformization of metric surfaces using isothermal coordinates},
        date={2022},
     journal={Ann. Fenn. Math.},
      volume={47},
      number={1},
       pages={155\ndash 180},
}

\bib{IwaMar01}{book}{
      author={Iwaniec, Tadeusz},
      author={Martin, Gaven},
       title={Geometric function theory and non-linear analysis},
      series={Oxford Mathematical Monographs},
   publisher={The Clarendon Press, Oxford University Press, New York},
        date={2001},
        ISBN={0-19-850929-4},
      review={\MR{1859913}},
}

\bib{IwaSve93}{article}{
      author={Iwaniec, Tadeusz},
      author={\v{S}ver\'{a}k, Vladim\'{\i}r},
       title={On mappings with integrable dilatation},
        date={1993},
        ISSN={0002-9939,1088-6826},
     journal={Proc. Amer. Math. Soc.},
      volume={118},
      number={1},
       pages={181\ndash 188},
         url={https://doi.org/10.2307/2160025},
      review={\MR{1160301}},
}

\bib{Kar07}{article}{
      author={Karmanova, M.~B.},
       title={Area and co-area formulas for mappings of the {S}obolev classes
  with values in a metric space},
        date={2007},
     journal={Sibirsk. Mat. Zh.},
      volume={48},
      number={4},
       pages={778\ndash 788},
}

\bib{Kir16}{article}{
      author={Kirsil\"{a}, Ville},
       title={Integration by parts on generalized manifolds and applications on
  quasiregular maps},
        date={2016},
        ISSN={1239-629X},
     journal={Ann. Acad. Sci. Fenn. Math.},
      volume={41},
      number={1},
       pages={321\ndash 341},
         url={https://doi.org/10.5186/aasfm.2016.4123},
      review={\MR{3467715}},
}

\bib{KaKoMa01}{article}{
      author={Kauhanen, Janne},
      author={Koskela, Pekka},
      author={Mal\'{y}, Jan},
       title={Mappings of finite distortion: discreteness and openness},
        date={2001},
        ISSN={0003-9527},
     journal={Arch. Ration. Mech. Anal.},
      volume={160},
      number={2},
       pages={135\ndash 151},
         url={https://doi.org/10.1007/s002050100162},
      review={\MR{1864838}},
}

\bib{KKMOZ03}{article}{
      author={Kauhanen, Janne},
      author={Koskela, Pekka},
      author={Mal\'{y}, Jan},
      author={Onninen, Jani},
      author={Zhong, Xiao},
       title={Mappings of finite distortion: sharp {O}rlicz-conditions},
        date={2003},
        ISSN={0213-2230},
     journal={Rev. Mat. Iberoamericana},
      volume={19},
      number={3},
       pages={857\ndash 872},
         url={https://doi.org/10.4171/RMI/372},
      review={\MR{2053566}},
}

\bib{LP20}{article}{
      author={Luisto, Rami},
      author={Pankka, Pekka},
       title={Sto\"{\i}low's theorem revisited},
        date={2020},
        ISSN={0723-0869,1878-0792},
     journal={Expo. Math.},
      volume={38},
      number={3},
       pages={303\ndash 318},
         url={https://doi.org/10.1016/j.exmath.2019.04.002},
      review={\MR{4149174}},
}

\bib{LWarea}{article}{
      author={Lytchak, Alexander},
      author={Wenger, Stefan},
       title={Area minimizing discs in metric spaces},
        date={2017},
     journal={Arch. Ration. Mech. Anal.},
      volume={223},
      number={3},
       pages={1123\ndash 1182},
}

\bib{LWintrinsic}{article}{
      author={Lytchak, Alexander},
      author={Wenger, Stefan},
       title={Intrinsic structure of minimal discs in metric spaces},
        date={2018},
     journal={Geom. Topol.},
      volume={22},
      number={1},
       pages={591\ndash 644},
}

\bib{Mei22}{article}{
      author={Meier, Damaris},
       title={Quasiconformal uniformization of metric surfaces of higher
  topology},
        date={2024},
     journal={Indiana Univ. Math. J.},
        note={To appear},
}

\bib{MalMar95}{article}{
      author={Mal\'{y}, Jan},
      author={Martio, Olli},
       title={Lusin's condition ({N}) and mappings of the class {$W^{1,n}$}},
        date={1995},
        ISSN={0075-4102,1435-5345},
     journal={J. Reine Angew. Math.},
      volume={458},
       pages={19\ndash 36},
         url={https://doi.org/10.1515/crll.1995.458.19},
      review={\MR{1310951}},
}

\bib{MN23}{article}{
      author={Meier, Damaris},
      author={Ntalampekos, Dimitrios},
       title={Lipschitz-volume rigidity and {S}obolev coarea inequality for
  metric surfaces},
        date={2024},
        ISSN={1050-6926,1559-002X},
     journal={J. Geom. Anal.},
      volume={34},
      number={5},
       pages={Paper No. 128, 30},
         url={https://doi.org/10.1007/s12220-024-01577-x},
      review={\MR{4718625}},
}

\bib{MR24}{article}{
      author={Meier, Damaris},
      author={Rajala, Kai},
       title={Definitions of quasiconformality on metric surfaces},
        date={2024},
       pages={preprint arXiv:2405.07476},
}

\bib{MW21}{article}{
      author={Meier, Damaris},
      author={Wenger, Stefan},
       title={Quasiconformal almost parametrizations of metric surfaces},
        date={2024},
     journal={J. Eur. Math. Soc.},
        note={published online first},
}

\bib{NR22}{article}{
      author={Ntalampekos, Dimitrios},
      author={Romney, Matthew},
       title={Polyhedral approximation and uniformization for non-length
  surfaces},
        date={2022},
       pages={preprint arXiv:2206.01128},
}

\bib{NR:21}{article}{
      author={Ntalampekos, Dimitrios},
      author={Romney, Matthew},
       title={Polyhedral approximation of metric surfaces and applications to
  uniformization},
        date={2023},
        ISSN={0012-7094,1547-7398},
     journal={Duke Math. J.},
      volume={172},
      number={9},
       pages={1673\ndash 1734},
         url={https://doi.org/10.1215/00127094-2022-0061},
      review={\MR{4608329}},
}

\bib{OZ08}{article}{
      author={Onninen, Jani},
      author={Zhong, Xiao},
       title={Mappings of finite distortion: a new proof for discreteness and
  openness},
        date={2008},
        ISSN={0308-2105},
     journal={Proc. Roy. Soc. Edinburgh Sect. A},
      volume={138},
      number={5},
       pages={1097\ndash 1102},
         url={https://doi.org/10.1017/S0308210506000825},
      review={\MR{2477453}},
}

\bib{Raj10}{article}{
      author={Rajala, Kai},
       title={Remarks on the {I}waniec-\v{S}ver\'{a}k conjecture},
        date={2010},
        ISSN={0022-2518},
     journal={Indiana Univ. Math. J.},
      volume={59},
      number={6},
       pages={2027\ndash 2039},
         url={https://doi.org/10.1512/iumj.2010.59.3946},
      review={\MR{2919747}},
}

\bib{Raj:17}{article}{
      author={Rajala, Kai},
       title={Uniformization of two-dimensional metric surfaces},
        date={2017},
     journal={Invent. Math.},
      volume={207},
      number={3},
       pages={1301\ndash 1375},
}

\bib{Res67}{article}{
      author={Re\v{s}etnjak, Ju.~G.},
       title={Spatial mappings with bounded distortion},
        date={1967},
        ISSN={0037-4474},
     journal={Sibirsk. Mat. \v{Z}.},
      volume={8},
       pages={629\ndash 658},
      review={\MR{215990}},
}

\bib{Ric93}{book}{
      author={Rickman, Seppo},
       title={Quasiregular mappings},
      series={Ergebnisse der Mathematik und ihrer Grenzgebiete (3) [Results in
  Mathematics and Related Areas (3)]},
   publisher={Springer-Verlag, Berlin},
        date={1993},
      volume={26},
        ISBN={3-540-56648-1},
         url={https://doi.org/10.1007/978-3-642-78201-5},
      review={\MR{1238941}},
}

\bib{RR19}{article}{
      author={Rajala, Kai},
      author={Romney, Matthew},
       title={Reciprocal lower bound on modulus of curve families in metric
  surfaces},
        date={2019},
        ISSN={1239-629X},
     journal={Ann. Acad. Sci. Fenn. Math.},
      volume={44},
      number={2},
       pages={681\ndash 692},
         url={https://doi.org/10.5186/aasfm.2019.4442},
      review={\MR{3973535}},
}

\bib{TY62}{article}{
      author={Titus, C.~J.},
      author={Young, G.~S.},
       title={The extension of interiority, with some applications},
        date={1962},
        ISSN={0002-9947},
     journal={Trans. Amer. Math. Soc.},
      volume={103},
       pages={329\ndash 340},
         url={https://doi.org/10.2307/1993663},
      review={\MR{137103}},
}

\bib{VG76}{article}{
      author={Vodop'yanov, S.~K.},
      author={Gol'dshtei, V.~M.},
       title={Quasiconformal mappings, and spaces of functions with first
  generalized derivatives},
        date={1976},
        ISSN={0037-4474},
     journal={Sibirsk. Mat. \v{Z}.},
      volume={17},
      number={3},
       pages={515\ndash 531, 715},
      review={\MR{414869}},
}

\bib{ManVil98}{article}{
      author={Villamor, Enrique},
      author={Manfredi, Juan~J.},
       title={An extension of {R}eshetnyak's theorem},
        date={1998},
        ISSN={0022-2518},
     journal={Indiana Univ. Math. J.},
      volume={47},
      number={3},
       pages={1131\ndash 1145},
         url={https://doi.org/10.1512/iumj.1998.47.1323},
      review={\MR{1665761}},
}

\bib{Wil:12}{article}{
      author={Williams, Marshall},
       title={Geometric and analytic quasiconformality in metric measure
  spaces},
        date={2012},
     journal={Proc. Amer. Math. Soc.},
      volume={140},
      number={4},
       pages={1251\ndash 1266},
}

\end{biblist}
\end{bibdiv}
\end{document}